\documentclass[12pt]{amsart}
\usepackage[latin1]{inputenc}
\usepackage{color}
\usepackage{bbm}
\usepackage{amsmath,amssymb}
\usepackage{enumerate}
\usepackage{mathrsfs}
\usepackage{verbatim}
\usepackage[top=2cm,bottom=2cm,left=3cm,right=3cm]{geometry}

\newtheorem{thm}{Theorem}[section]
\newtheorem{cor}[thm]{Corollary}
\newtheorem{lem}[thm]{Lemma}
\newtheorem{prop}[thm]{Proposition}
\theoremstyle{definition}
\newtheorem{defin}[thm]{Definition}

\newtheorem{rem}[thm]{Remark}
\newtheorem{example}[thm]{Example}

\numberwithin{equation}{section}

\date{\today}

\newcommand{\lin}{\operatorname{span}}

\newcommand{\dif}{\,\mathrm{d}}

\DeclareMathOperator{\card}{card}

\newcommand{\charfun}{\ensuremath{\mathbbm 1}} % characteristic function
\newcommand{\sm}[1]{\ensuremath{#1'}}  % smallest child
\newcommand{\la}[1]{\ensuremath{#1''}} % largest child
\allowdisplaybreaks

\begin{document}
\title[Properties of local orthonormal systems: Unconditionality in $L^p$]{Properties of local orthonormal systems \\
Part I:
Unconditionality in $L^p$, $1<p<\infty$}
\author[J. Gulgowski]{Jacek Gulgowski}
\address{Faculty of Mathematics, Physics and Informatics, University of Gda\'nsk,
ul. Wita Stwosza 57, 80-308 Gda\'nsk , Poland}
\email{ jacek.gulgowski@ug.edu.pl}

\author[A.Kamont]{Anna Kamont }
\address{Institute of Mathematics, Polish Academy of Sciences, Branch in Gda\'nsk, ul. Abrahama 18, 81-825 Sopot, Poland}
\email{Anna.Kamont@impan.pl}

\author[M. Passenbrunner]{Markus Passenbrunner}
\address{Institute of Analysis, Johannes Kepler University Linz, Austria, 4040 Linz, Altenberger Strasse 69}
\email{markus.passenbrunner@jku.at, markus.passenbrunner@gmail.com}


\begin{abstract}
	Assume that we are given a filtration $(\mathscr F_n)$ on a probability space
	$(\Omega,\mathscr F,\mathbb P)$ of the form that each $\mathscr F_n$ is generated
	by the partition of one atom of $\mathscr F_{n-1}$ into two atoms of $\mathscr F_n$
	having positive measure. Additionally, assume  that we are given a finite-dimensional linear space $S$ of
	$\mathscr F$-measurable, bounded functions on $\Omega$ so that on each 
	atom $A$ of any $\sigma$-algebra $\mathscr F_n$, all $L^p$-norms of functions
	in $S$ are comparable independently of $n$ or $A$.
	Denote by
	$S_n$ the space of functions that are given locally, on atoms of $\mathscr F_n$,
	by functions in $S$ and by $P_n$ the orthoprojector (with respect to the inner 
	product in $L^2(\Omega)$) onto $S_n$.

	Since $S = \lin\{\charfun_\Omega\}$ satisfies the above assumption and $P_n$ is 
	then the conditional expectation $\mathbb E_n$ with respect to $\mathscr F_n$, for such 
	filtrations, martingales $(\mathbb E_n f)$ are special cases of our setting.
	
	We show in this article that certain convergence results
	that are known for martingales (or rather martingale differences) are also 
	true in the general framework described above. More precisely, we show that   
	the differences $(P_n - P_{n-1})f$ form an unconditionally convergent series and are 
	democratic in $L^p$ for $1<p<\infty$. This implies that those differences 
	form a greedy basis in $L^p$-spaces for $1<p<\infty$.
\end{abstract}
\maketitle

\section{Introduction}\label{sec:intro}
Let $(\Omega,\mathscr F,\mathbb P) = (\Omega,\mathscr F,|\cdot|)$ be a probability space
and let $(\mathscr F_n)_{n\geq 0}$ 
be a filtration in $\mathscr F$. The space $L^p(\Omega,\mathscr F,\mathbb P)$ will further be denoted by $L^p$. For $1<p<\infty$, given a function $f\in L^p$ and denoting by $\mathbb E_n$
the conditional expectation of $f$ with respect to the $\sigma$-algebra $\mathscr F_n$, the 
martingale differences $(\mathbb E_n - \mathbb E_{n-1})f$ are orthogonal to each 
other and the series of these differences $\sum_{n\geq 1}(\mathbb E_n - \mathbb E_{n-1})f$  converge unconditionally in $L^p$ by D.~Burkholder's inequality. 
We use \cite{Pisier2016} as our reference for the latter inequality and for other basic facts about martingales.
Assume now that $(\mathscr F_n)_{n\geq 0}$ is a \emph{binary} filtration, meaning that
\begin{enumerate}
	\item $\mathscr F_0 = \{\emptyset, \Omega\}$,
\item for each $n\geq 1$,
$\mathscr F_{n}$ is generated by $\mathscr F_{n-1}$ and the
subdivision of exactly one atom $A_n$ of $\mathscr F_{n-1}$ into two atoms
$\sm{A_n},\la{A_n}$ of $\mathscr F_{n}$ satisfying $|\la{A_n}| \geq |\sm{A_n}|>0$. 
\end{enumerate}
For notational purpose, set $\mathscr F_{-1} := \mathscr F_0 = \{\emptyset, \Omega\}$ and $A_0 = \Omega$.
Moreover, denote the collection of all atoms in $\mathscr F_n$ by $\mathscr A_n$ and define $\mathscr A
 = \cup_n \mathscr A_n$.
 In this case, the range $S_n$ of the conditional expectation $\mathbb E_n$ consists of the 
 piecewise constant functions. This precisely means that  $f\in S_n$ if and only if for all $A\in\mathscr A_n$
 there exists $g\in \lin\{\charfun_\Omega\}$ with $\charfun_A f = \charfun_A g$.
 Here we denote by $\charfun_A$ the characteristic function of the 
set $A$  given by $\charfun_A(x)= 1$ if $x\in A$ and $\charfun_A(x)=0$ otherwise.
Observe that clearly, for each atom $A\in\mathscr A$, we have
\begin{equation}\label{eq:const}
	|\{ \omega\in A : |f(\omega)| = \|f\|_{A}
	\}| = |A|,\qquad f\in\lin\{ \charfun_\Omega\},
\end{equation}
where by $\|f\|_A$ we denote the $L^\infty$-norm of $f$ on the set $A$.

We now switch our viewpoint and assume that instead of $\lin\{\charfun_\Omega\}$, 
we are given an arbitrary vector space $S\subseteq L^\infty$ 
of finite dimension $k$.
Instead of \eqref{eq:const}, we assume  for each $A\in\mathscr A$ the inequality
	\begin{equation}\label{eq:L1Linfty}
			|\{ \omega\in A : |f(\omega)| \geq c_1\|f\|_{A}
	\}| \geq c_2|A|,\qquad f\in S
\end{equation}
for some constants $c_1,c_2\in (0,1]$ independent of $A$.
Inequality \eqref{eq:L1Linfty} for fixed $A$ is 
equivalent to the comparability of $\|f\|_A$ and the mean value $|A|^{-1}\int_A|f|\dif\mathbb P$
for all $f\in S$. 
Indeed, if we set $U :=\{ \omega\in A : |f(\omega)| \geq c_1\|f\|_{A} \}$,
this can be seen by splitting the integral $\int_A |f|\dif\mathbb P$ into 
the two integrals over $U$ and $A\setminus U$, using the available 
estimates, and choosing the respective constants accordingly.

For some explicit examples of spaces $S$ and collections of atoms $\mathscr A$
 satisfying this inequality, we refer to Section~\ref{sec:examples}. 

Given any subspace $V\subset L^1$ and any measurable set $A\in\mathscr F$, we let 
\[
	V(A) = \{ f\cdot \charfun_A : f\in V\}.	
\]
Let $S_{-1} = \{0\}$ and, for each $n\geq 0$,
\[
S(\mathscr F_n):=S_n := \bigoplus_{A\in\mathscr A_n} S(A).
\]
Thus, $S(\mathscr F_n)=S_n$ consists of those functions that are piecewise (on each atom of $\mathscr F_n$)
 contained in $S$.
Letting $P_n$ be the orthoprojector onto $S_n$ (with respect to the inner product 
on the space $L^2$), we show in this article 
that the series $\sum_{n\geq 1}(P_n - P_{n-1})f$ of corresponding differences  also converge unconditionally in $L^p$
for $f\in L^p$, $1<p<\infty$. 

We even show a more general result. To describe it, 
observe first  that the dimension of $S_n / S_{n-1}$ generally is bigger than one.
Therefore, we consider the following setting.
For each $n\geq 0$, we assume that we are given 
a strictly increasing sequence of spaces $S_{n-1} =
V_{n,0} \subset V_{n,1} \subset \cdots \subset V_{n,\ell_n} = S_{n}$,
such that $\dim V_{n,j} / V_{n,j-1} = 1$ for each $j=1,\ldots,\ell_n$ where $\ell_n = \dim S_n - \dim S_{n-1}$,
which clearly satisfies $0\leq \ell_n\leq k$.
Denote by $P_{n,j}$ the orthoprojector onto $V_{n,j}$.
We re-enumerate the spaces $(V_{n,j})$ for $n\geq 0$ and $0\leq j\leq \ell_n - 1$
 so that we have a strictly increasing sequence of spaces $(V_m)$. 
 Denote by $R_m$ the orthoprojector onto
$V_m$.
For an integrable function $f\in
L^1(\Omega)$ and for an index $m$, set
\[
	f_m := R_m f, \qquad df_m = f_m - f_{m-1}.
\]

Then we show that the series of these differences converge unconditionally in $L^p$ for $1<p<\infty$ 
independently of the choice of spaces $(V_{n,j})$.
\begin{thm}\label{thm:uncond_intro}
	Let $(\Omega,\mathscr F,\mathbb P)$ be a probability space, $(\mathscr F_n)$
	a binary filtration and $S$ a finite dimensional linear space of $\mathscr F$-measurable
	functions on $\Omega$ satisfying \eqref{eq:L1Linfty} for every atom $A\in\mathscr A$.

	Then, for every $1<p<\infty$ there exists a positive constant $C_p$ such that for every 
	$f\in L^p$ and all positive integers $n$,
	\[
		\big\| \sum_{m\leq n} \pm df_m \big\|_p \leq C_p \| f\|_p.		
	\]
\end{thm}

The result of Theorem~\ref{thm:uncond_intro} should be compared to corresponding 
results \cite{GevorkyanKamont2004,Passenbrunner2014} about the unconditionality (in $L^p[0,1]$ for $1<p<\infty$) of univariate orthogonal spline functions 
for arbitrary partitions.
(In the piecewise linear case, those orthogonal spline functions 
are called general Franklin functions.)
A necessary prerequisite for  deriving unconditionality of those functions are sharp 
pointwise estimates for those orthogonal spline functions. Due to the fact that 
the considered splines admit smoothness conditions at the breakpoints,
orthogonal spline functions are not localized, which makes them hard to estimate. 
On the other hand, note that discontinuous spline systems, i.e., such that each knot is of maximal admissible multiplicity, are examples of systems treated in this article. 

Our setting in the present article considers differences $df_m$ that have
the same localization property as  Haar systems. This allows us to prove 
general pointwise estimates for $df_m$, which is done in  Section \ref{sec:pw}.
As a consequence, we  
prove Theorem~\ref{thm:uncond_intro} under very mild 
assumptions on the domain of the involved functions.
Here, we work with arbitrary probability spaces compared to some bounded interval, equipped 
with Lebesgue measure, that 
is needed for univariate orthogonal spline functions.

Denote by $\psi_m^p$ a function with $\|\psi_m^p\|_p = 1$, $\psi_m^p\in V_m$ and $\psi_m^p\perp V_{m-1}$.
Then, we additionally show that the functions $(\psi_m^p)$ are 
democratic in $L^p$ for $1<p<\infty$. 
\begin{thm}\label{thm:democracy}For each $1<p<\infty$,
there exists a positive  constant $C_p$ depending on $p$, the constants $c_1,c_2\in(0,1]$ given in \eqref{eq:L1Linfty}, and the dimension of $S$ so that 
\[
C_p^{-1} (\card\Lambda)^{1/p}\leq \big\| \sum_{m\in\Lambda} \psi_m^p \big\|_p \leq C_p (\card\Lambda)^{1/p}
\]
for every finite subset $\Lambda$ of indices, where $\card \Lambda$ denotes the cardinality of $\Lambda$.
\end{thm}

If we combine the results of Theorems~\ref{thm:uncond_intro} and \ref{thm:democracy},
by a result of S.~Konyagin and V.~Temlyakov \cite{KonyaginTemlyakov1999} (see also the book \cite{Temlyakov2011} by V.~Temlyakov), 
$(\psi_j^p)$ is also a greedy basis in its closed span $U_p=\overline{\operatorname{span}}_{L^p}(\cup_n S_n)$ in $L^p$, meaning that
the best $m$-term approximation of each function $f = \sum_j c_j \psi_j^p\in U_p$ in terms of 
the functions $(\psi_j^p)$ is achieved (up to a multiplicative constant) by greedy approximation, i.e.,
by taking just the terms in the expansion of $f$ with the $m$ highest absolute 
values of coefficients $|c_j| = \|\langle f,\psi_j\rangle\psi_j\|_p$.

The main motivation of the work presented in this article comes from the following question. 
It is known that the best $n$-term Haar approximation of a function $f \in L^p[0,1]^d$, $1 <p< \infty$ 
can be achieved by greedy approximation. 
On the other hand, P. Petrushev showed that the spaces corresponding to the best 
$n$-term approximation by linear combinations of characteristic functions of 
 dyadic cubes and by linear combinations of Haar functions coincide, 
see  \cite[Theorems 3.3 and 5.3]{pp.2003.a}. 

 In \cite{part2}, we ask the question whether this is also true in the 
 more general situation given in the present article. 
 Interesting special cases include  general Haar systems or  
orthonormal piecewise polynomial systems corresponding to rectangular 
partitions, as discussed in Example~\ref{ex:multivariate_remez}. 
 However, for this we need to know that the orthonormal systems under 
 consideration are greedy basic sequences in $L^p$. 
 This property is known (or can be easily deduced from the unconditionality) 
  for general Haar systems (cf. \cite{ak.2001}), or for orthonormal spline systems 
  on the unit interval (cf. \cite{GevorkyanKamont2004,Passenbrunner2014}).
   The required result in general is provided by Theorems \ref{thm:uncond_intro} 
   and \ref{thm:democracy} above.

The organization of the article is as follows. Section~\ref{sec:examples} lists 
some explicit examples of spaces $S$ and atoms $\mathscr A$ satisfying inequality~\eqref{eq:L1Linfty}.
In Section~\ref{sec:density} we give a criterion under which $\cup_n S_n$ is dense
in $L^p$. In Section~\ref{sec:pw} we prove some crucial pointwise estimates for $P_nf$
that are used subsequently. Sections~\ref{sec:decomp} and \ref{sec:uncond} show Theorem~\ref{thm:uncond_intro}
by deriving a decomposition of $L^1$-functions similar to Gundy's decomposition 
for martingales. In Section~\ref{sec:democracy} we show Theorem~\ref{thm:democracy}. 
Section~\ref{sec:infty} treats the case of infinite measure spaces.
In the last Section~\ref{sec:more_than_two} we investigate the 
questions of unconditionality and democracy in the 
setting of non-binary filtrations, meaning that we consider splitting $A_{n}\in \mathscr A_{n-1} $ into 
$r\geq 2$ atoms $A_{n,1},\ldots,A_{n,r}$ of $\mathscr F_n$ instead of just two.
In particular, we show that general results as Theorem~\ref{thm:uncond_intro} 
that show unconditionality independently of the choice of $(V_{n,j})$,
are not possible in this case.

\section{Examples}\label{sec:examples}
In this section we give a list of explicit examples of probability spaces
$(\Omega,\mathscr F,\mathbb P)$, finite-dimensional vector spaces $S$ and 
binary filtrations
$(\mathscr F_n)$ satisfying inequality~\eqref{eq:L1Linfty} for some constants $c_1,c_2 \in(0,1]$.

Before we discuss examples, we make the following simple observation:
\begin{lem}\label{lem:remez_general}
	Let $(\Omega,\mathscr F,\mathbb P)$ a probability space with $|\cdot|=\mathbb P$ and $(\mathscr F_n)$ be 
	a binary filtration. Let $S$ be chosen such that there exist
	two positive constants $C,n$ so that for all $A\in \mathscr A$ we have the inequality
	\begin{equation}\label{eq:remez_general}
		\|f\|_A \leq \Big(\frac{C |A|}{|E|}\Big)^n \| f\|_E,\qquad f\in S,
	\end{equation}
	for all $\mathscr F$-measurable subsets $E\subset A$.

	Then, inequality~\eqref{eq:L1Linfty} is 
	satisfied for all $A\in\mathscr A$ with the constants $c_1 = (2C)^{-n}$ and $c_2 = 1/2$.
\end{lem}
\begin{proof}
	Let $A\in \mathscr A$ be a fixed atom and let $f\in S$ be arbitrary.
	We assume without restriction that $\|f\|_A >0$.
	Then, we just need to apply assumption~\eqref{eq:remez_general} to the 
	$\mathscr F$-measurable set 
	\[
		E = \{ \omega\in A : |f(\omega)|< \alpha \|f\|_A \},
	\]	
	with $\alpha = (2C)^{-n}$. Indeed, doing that implies
	\[
		\| f\|_A  \leq \Big(\frac{C |A|}{|E|}\Big)^n \| f\|_E \leq 
		\alpha \Big(\frac{C |A|}{|E|}\Big)^n \|f\|_A.
	\]
	Since $\|f\|_A > 0 $ we obtain
		$|E|^n \leq \alpha (C |A|)^n$ and the specific value of $\alpha$ yields $|E| \leq |A|/2$.
		Therefore,
		\[
			|A\setminus E| = |\{ \omega \in A : |f(\omega)| \geq \alpha\|f\|_A \}| \geq \frac{|A|}{2}.
		\]
		Since $A\in\mathscr A$ and $f\in S$ were arbitrary, this shows \eqref{eq:L1Linfty}
		with the constants $c_1 = \alpha = (2C)^{-n}$ and $c_2 = 1/2$.
\end{proof}

Thus, if we give examples satisfying \eqref{eq:remez_general},  in particular we
have \eqref{eq:L1Linfty} as well.

In the following two examples (Examples~\ref{ex:poly1} and \ref{ex:turan1}), 
let $\Omega$ be a bounded subinterval of $\mathbb R$, $\mathscr F$ be the
	Borel $\sigma$-algebra, $\lambda$ the (one-dimensional) Lebesgue 
	measure and $\mathbb P = |\cdot|$  the renormalized Lebesgue measure $\lambda/\lambda(\Omega)$.
\begin{example} \label{ex:poly1}
Let $S$ be the space of polynomials of 
	degree at most $n$. Then, the classical Remez inequality 
	(see for instance the paper \cite{Bojanov1993} by B. Bojanov)
	implies that for
	every $f\in S$ and every
	 bounded interval $A$ and every measurable subset $E\subset A$ we have 
	\[
		\|	f \|_A  \leq \Big( \frac{4 |A|}{|E|} \Big)^n \| f\|_E.
	\]
	Thus, for every binary filtration $(\mathscr F_n)$,
	where all atoms $A\in\mathscr A$ are intervals, we use Lemma~\ref{lem:remez_general}
	to deduce inequality~\eqref{eq:L1Linfty} with the constants $c_1 = 8^{-n}$ 
	and $c_2 = 1/2$.
\end{example}

\begin{example} \label{ex:turan1}
Let $(\lambda_k)_{k=1}^n$ be arbitrary complex numbers. 
In what follows $\Re z$ denotes the real part of the complex number $z$.
Set
\[
	S := \lin \{ \exp (\lambda_k \cdot) : k = 1,\ldots,n\}.	
\]
F.~Nazarov's extension \cite{Nazarov1993} of P.~Turan's lemma  states that 
for every $f\in S$, every bounded interval $A$ and every measurable subset $E\subset A$ 
we have
\begin{equation}\label{eq:nazarov_turan_pre}
	\|f\|_A \leq \exp\Big(\max_{1\leq k\leq n} |\Re\lambda_k|\cdot \lambda(A)\Big)	 
	\Big( \frac{C |A|}{|E|}\Big)^{n-1} \|f\|_E
\end{equation}
for some positive  constant $C$. Letting $C_\Omega = 
\exp\big( \max_{1\leq k\leq n} |\Re \lambda_k|\cdot \lambda(\Omega) / (n-1) \big) C$, we have
as a consequence 
\begin{equation}\label{eq:nazarov_turan}
	\| f\|_A \leq \Big( \frac{C_\Omega |A|}{|E|}  \Big)^{n-1} \|f\|_E.
\end{equation}
Inequality~\eqref{eq:nazarov_turan} shows 
assumption~\eqref{eq:remez_general} of Lemma~\ref{lem:remez_general}. 
Therefore, Lemma~\ref{lem:remez_general} implies inequality~\eqref{eq:L1Linfty} with
the constants $c_1 = (2C_\Omega)^{-n}$ and $c_2 = 1/2$.

In particular, \eqref{eq:L1Linfty} is satisfied if $S$ is the space of trigonometric polynomials
up to some degree $n$.
\end{example}

Now, let $\Omega\subset \mathbb R^d$ be a convex body (i.e., a compact 
convex subset of $\mathbb R^d$ with nonempty interior), $\mathscr F$
the Borel $\sigma$-algebra on $\Omega$ and $\mathbb P = |\cdot|$ be the 
$d$-dimensional Lebesgue measure renormalized so that $|\Omega| = 1$.

\begin{example}\label{ex:multivariate_remez}
	If $p(x) = \sum_{\alpha\in I} a_\alpha x^\alpha$ is a $d$-variate polynomial
	where $I$ is a finite set containing $d$-dimensional multiindices, then
	the degree of $p$ is defined as $\max \{ \sum_{i=1}^d \alpha_i : \alpha\in I\}$.
	For a fixed non-negative integer $n$, let $S$ be the space of polynomials
	of degree at most $n$.

	Then, the multivariate Remez inequality (see \cite{Ga2001,BrudnyiGanzburg1973}) states
	that if $A\subset \mathbb R^d$ is a convex body and $E\subset A$ a measurable subset,
	for all $f\in S$ we have 
	\begin{equation}\label{eq:remez_multi}
		\|f\|_A \leq \Big(  \frac{4d |A|}{|E|}\Big)^n \|f\|_E.
	\end{equation}

	Thus, for every binary filtration $(\mathscr F_n)$,
	where all atoms $A$ are convex sets whose closure $\bar{A}$ has
	the same Lebesgue measure as $A$, we use Lemma~\ref{lem:remez_general} to 
	deduce inequality \eqref{eq:L1Linfty} with the constants $c_1 = (8d)^{-n}$ 
	and $c_2 = 1/2$.
\end{example}

We remark here that similar to the multivariate Remez inequality as in 
Example~\ref{ex:multivariate_remez}, there are
multivariate versions of \eqref{eq:nazarov_turan}, see for instance the result 
\cite{Fontes-Merz2006} by N.~Fontes-Merz.
Therefore, by Lemma~\ref{lem:remez_general}, such spaces of functions also
satisfy inequality~\eqref{eq:L1Linfty} and are thus examples of the general
framework discussed here.

\begin{example}
	Let $\gamma : [0,1]\to \mathbb C$ be a curve of unit length that is two times continuously 
	differentiable and let $S$ be the space of polynomials of degree at most $n$.
	Let $\mathbb P$ be the one-dimensional Lebesgue measure on the image of $\gamma$.
	Then, there exist constants $c_1,c_2\in (0,1]$ that may depend on $\gamma$ such that,
	for all intervals $I\subset [0,1]$ and denoting by  $\gamma_I$ 
	the restriction of $\gamma$ to the interval $I$, we have the inequality 
	\[
			\mathbb P( \{ z\in \gamma_I : |f(z)| \geq c_1\|f\|_{\gamma_I} \}) \geq c_2 |\gamma_I|,\qquad f\in S.
	\]
	
	This inequality is easy to see if $\gamma$ is a straight line and if $\gamma$ is 
	not a straight line, we approximate $\gamma$ by straight lines and use 
	translation and dilation invariance of the class of polynomials of degree at most  $n$.
\end{example}

\begin{rem}
	 If inequality~\eqref{eq:L1Linfty} is satisfied for some
	vector space $S$, any subspace of $S$ clearly satisfies \eqref{eq:L1Linfty}
	as well. Thus, in each of the above examples it is possible to replace the 
	space $S$ considered there by any linear subspace thereof.
\end{rem}

\section{Density of $\cup_n S_n$ in $L^p$ spaces}\label{sec:density}
Here and in the following we assume that $(\Omega,\mathscr F,\mathbb P)$ is 
a probability space, $(\mathscr F_n)$ is a binary filtration and 
$S\subset L^\infty$ is a vector space satisfying
inequality \eqref{eq:L1Linfty} for each atom $A\in\mathscr A$.
We also use the notation introduced in Section~\ref{sec:intro}.

Next, we  give a condition in terms 
of the filtration $(\mathscr F_n)$ under which 
$\cup_n S_n$ is dense in $L^p$. Here we also have to assume something about the space $S$, 
but as we will see in the remark below, this condition is necessary.

\begin{thm}\label{thm:dense_suff}
	Let $0< p<\infty$  and assume that the following two conditions are
satisfied:
\begin{enumerate}
	\item  $\big| \bigcap_{f\in S} \{ f = 0\} \big| = 0$,
	\item  for all $\varepsilon>0$ and every $A\in \mathscr F$ there exists a
		positive integer $N$ and a set $B\in \mathscr F_N$ satisfying
		$|A\Delta B|\leq \varepsilon$.
\end{enumerate}

Then, $\cup_n S_n$ is dense in $L^p$.
\end{thm}
\begin{rem}
	(a) After the  choice of a basis $p_1,\ldots,p_k$ of $S$, we have 
	\[
			\bigcap_{f\in S} \{ f = 0\} = \bigcap_{i=1}^k \{ p_i = 0\},
	\]
	which in particular implies the measurability of the set $\bigcap_{f\in S} \{ f = 0\} $.

	(b) Condition (1) is necessary: for assume that condition (1) is not satisfied. Let 
	$D = \cap_{f\in S}\{f = 0\}\in \mathscr F$,
	with $|D|=c>0$. By definition of $S_n$, every function $f\in \cup_n S_n$ vanishes identically
	on $D$, thus
	\[
		c = \inf_{f\in \cup_n S_n}	\| \charfun_D - f\|_{L^p(D)}^p 
		\leq \inf_{f\in \cup_n S_n} \|\charfun_D - f\|_p^p.
	\]	
	This implies that $\cup_n S_n$ is not dense in $L^p$.

	(c) If the $\sigma$-algebra $\mathscr F$ is generated by $\cup_n \mathscr F_n$, then condition
	(2) is satisfied.

\end{rem}

\begin{proof}[Proof of Theorem~\ref{thm:dense_suff}] 
	Choose a basis $p_1,\ldots,p_k$ of $S$ and assume that
	(2) is satisfied. Observe that also, by (1), we have $|\cap_{i=1}^k \{p_i
	= 0\} |=0$.
	Let $A\in\mathscr F$.
	We intend to approximate $\charfun_A$ in $L^p$ by functions in $\cup_n
	S_n$.
	Fix $\varepsilon>0$. 
	First, we split $\mathbb C\setminus \{0\}$ into disjoint cubes $(I_j)_{j\in\mathbb
		Z\setminus\{0\}}$ so that for each $j$, and every $x,y\in I_j$,
		we have that $x/y \in B(1,\varepsilon):=\{ z \in \mathbb C : |z-1| < \varepsilon\}$.
	For $i = 1,\ldots,k$ and $j\in \mathbb Z\setminus \{0\}$, define the sets
	\[
		C_{i,j} = \{ p_i \in I_j\}
	\]
	and $C_{i,0}=\{ p_i = 0\}$.
	For $j = (j_1,\ldots,j_k)\in\mathbb Z^k$, define the disjoint sets
	\[
		D_j := A \cap C_{1,j_1}\cap \cdots \cap C_{k,j_k}.
	\]
	Let $W$ be a finite set of indices $j\in \mathbb Z^k\setminus \{0\}$ so that $|\cup_{j\in W} D_j| \geq 
	(1-\varepsilon)|A|$ and for each $j\in W$ we have that $|D_j|>0$. 
	This is possible for every $\varepsilon>0$ since by assumption (1), we have $|D_0| = 0$.
	For each $j\in W$, choose a set $B_j \in \mathscr F_{N_j}$
	with
	\[
		| B_j \Delta D_j| \leq \varepsilon_j,
	\]
	where $\varepsilon_j$ depending on $\varepsilon,p,k$ will be chosen later.
	Since $B_j\in \mathscr F_{N_j}$, it can be written as a finite union of
	disjoint atoms $(A_{j,\ell})_\ell$ from $\mathscr F_{N_j}$.

	For $q = 1-c_2/2 \in (0,1)$ with $c_2$ being the constant in \eqref{eq:L1Linfty},
	 split the indices $\ell$ into
	\[
		\Gamma_j := \{\ell : |A_{j,\ell} \cap D_j|\leq q |A_{j,\ell}|\},\qquad
		\Lambda_j := \{\ell : |A_{j,\ell} \cap D_j| >
		q|A_{j,\ell}|\}.
	\]
	We have the following estimate
	\begin{equation}\label{eq:e_j}
		\varepsilon_j \geq |B_j \setminus D_j| = \sum_\ell
		|A_{j,\ell}\setminus D_j| \geq \sum_{\ell\in\Gamma_j}
		|A_{j,\ell}\setminus D_j| = \sum_{\ell\in \Gamma_j}
		|A_{j,\ell}\cap D_j^c| \geq (1-q) \sum_{\ell\in \Gamma_j}
		|A_{j,\ell}|.
	\end{equation}
	For $j\in W$, let $u_j := p_i$ with $i\in\{1,\ldots,k\}$ be chosen so that
	$j_i \neq 0$. 
	Let $\omega_j\in D_j$ be arbitrary.
	Then, set $v_j := u_j / u_j(\omega_j)$ and 
	$E_j := \cup_{\ell\in \Lambda_j} A_{j,\ell}$.
	Observe that \eqref{eq:e_j} then implies
	\begin{equation}\label{eq:e_j_2}
		| E_j\Delta D_j | \leq |B_j\Delta D_j| + \Big| \bigcup_{\ell\in
		\Gamma_j} A_{j,\ell}\Big| \leq \big(1+ (1-q)^{-1}\big)
		\varepsilon_j.
	\end{equation}
	Now, we define the approximation $f$ of $\charfun_A$ to be a suitable finite sum of the
	following form:
	\[
		f = \sum_{j\in W}  v_j \charfun_{E_j} \in \cup_n S_n.
	\]
	The values of the function $v_j$ on
	$D_j$ are
	contained in the ball $B(1,\varepsilon)$ by construction of the
	sets $D_j$ and the intervals $I_j$. 

	Next we see that $|v_j|$ is bounded on $E_j$ by $c_1^{-1}(1+\varepsilon)$. Assume the contrary, i.e.,
	 assume that $\| v_j \|_{A_{j,\ell}} > (1+\varepsilon)/c_1$ for some
	$\ell\in \Lambda_j$.
	This would imply using \eqref{eq:L1Linfty} that
	\begin{equation}\label{eq:c3_1}
		c_2|A_{j,\ell}| \leq |\{ \omega\in A_{j,\ell} : |v_j| \geq
		c_1 \|v_j\|_{A_{j,\ell}}\}| \leq |\{ \omega\in
			A_{j,\ell} : |v_j| > 1+\varepsilon\}|.
	\end{equation}
	On the other hand, since $\ell\in \Lambda_j$,
	\begin{equation}\label{eq:c3_2}
		q|A_{j,\ell}| \leq |D_j \cap A_{j,\ell}| \leq |\{ \omega\in
			A_{j,\ell} : |v_j|\leq 1+\varepsilon\}|.
	\end{equation}
	Both \eqref{eq:c3_1} and \eqref{eq:c3_2} simultaneously are impossible by 
	definition of $q= 1-c_2/2$ and thus we obtain $\|v_j\|_{L^\infty(E_j)} \leq (1+\varepsilon)/c_1$.

	Since we know the inequality $|\cup_{j\in W} D_j| \geq (1-\varepsilon)|A|$, 
	in order to estimate $\| f - \charfun_A\|_p$, it suffices to estimate
	the $p$-norm of the function $\delta= f-\sum_{j\in W} \charfun_{D_j}$.
	Define $\delta_1 = \sum_{j\in W} (1-v_j)\charfun_{D_j\cap E_j}$ and
	$\delta_2 = \delta - \delta_1$.
Since the range of $v_j$ on $D_j$ is contained in $B(1,\varepsilon)$ and the
sets $D_j\cap E_j$ are disjoint, we obtain $\|\delta_1\|_p \leq \varepsilon$.
Moreover, we have the pointwise inequality $|\delta_2| \leq \sum_{j\in W}
(1+\|v_j\|_{L^\infty(E_j)}) \charfun_{D_j\Delta E_j} \leq \sum_{j\in W}
(1+(1+\varepsilon)c_1^{-1}) \charfun_{D_j\Delta E_j}$. By \eqref{eq:e_j_2}, this implies in the
case $p\geq 1$
\[
\|\delta_2\|_p \leq (1+(1+\varepsilon)c_1^{-1})\big(1+ (1-q)^{-1}\big)^{1/p}\sum_{j\in W}
\varepsilon_j^{1/p},
\]
and in the case $p<1$
\[
	\|\delta_2\|_p^p \leq (1+(1+\varepsilon)c_1^{-1})^p \big( 1+(1-q)^{-1}\big) \sum_{j\in
	W}\varepsilon_j.
\]
In both cases, we can choose $(\varepsilon_j)$ depending on $\varepsilon,p,k$ such
that $\|\delta_2\|_p\leq \varepsilon$. Combining this with the estimate for
$\|\delta_1\|_p$, we obtain that the characteristic function $\charfun_A$ of
each set $A\in \mathscr F$ can be approximated arbitrarily closely by functions
in $\cup_n S_n$ in $L^p$, which implies the conclusion.
\end{proof}

\section{A pointwise bound for $P_{n,j}$}
\label{sec:pw}
We continue to use the notation from Section~\ref{sec:intro}. In particular,
$P_n$ denotes the orthoprojector onto $S_n$.
Since $S \subset L^\infty(\Omega)$,
the operators $P_n$ can be extended to $L^1(\Omega)$ and
it is easy to see that for each $n$ it satisfies
\begin{equation} \label{eq:commute}
P_n(\charfun_B f) = \charfun_B P_n f,\qquad B\in \mathscr F_n, f\in L^1.
\end{equation}
Indeed, if $g\in S_n$, we have
\[
	\int (\charfun_B f - \charfun_B P_n f) \bar{g}\dif\mathbb P = \int (f - P_n
	f)\overline{\charfun_B g}\dif\mathbb P,
\]
which vanishes since $\charfun_B g$, together with $g$, is contained in $S_n$,
which is the range of $P_n$.
If $n,j$ are chosen so that $V_{n,j} = V_m$, we have of course $P_{n,j} = R_m$.
Define the index set
$I_n = \{ m : S_{n-1} \subsetneq V_m \subseteq S_n\}$.

\begin{lem}\label{lem:orth}
	Suppose that $n\geq 1$ and let $g\in S_n$ with $g\perp S_{n-1}$. Then
	$g$ vanishes identically a.e. on the complement $A_n^c$ of $A_n$.
\end{lem}
\begin{proof}
	Let $A = A_n^c$. Then, $\charfun_A g \in S_{n-1}$, which implies that
	\[
		0= \langle g, \charfun_A g\rangle = \int_A |g|^2 \dif\mathbb P,
	\]
	which gives that $g$ vanishes a.e. on $A=A_n^c$.
\end{proof}

\begin{rem}\label{rem:Pm}
	Let $n\geq 1$.
Since for $m\in I_n$, the function $df_m$ is orthogonal to
$S_{n-1}$, this implies that 
the support of $df_m$ is a subset of $A_n$ for
$m\in I_n$.

Also this implies that for each $B\in\mathscr F_{n-1}$ and each 
$m\in I_n$ and every $f\in L^1$, we have $R_m(\charfun_B f) =
\charfun_B R_m f$.
\end{rem}

We use the symbol $A(t)\lesssim B(t)$ in order to denote the fact that there
exists a positive  constant $C$ depending only on $c_1,c_2$ from 
\eqref{eq:L1Linfty} and $k$  so that for all $t$
we have the inequality $A(t) \leq C B(t)$, where $t$ denotes all implicit or
explicit dependencies that the objects $A$ and $B$ might have. Similarly, we use
the notations $A(t)\gtrsim B(t)$ and $A(t)\simeq B(t)$.

	\begin{lem}\label{lem:ess_estimate}
		Let $n\geq 1$ and let $\psi\in S_n$ with $\psi\perp S_{n-1}$ and $\|\psi\|_2=1$.

		Then we have the estimates 
		\[
			\|\psi\|_{A_n^c} = 0,\qquad \|\psi\|_{A_n'}\lesssim |A_n'|^{-1/2}
			,\qquad \|\psi\|_{A_n''} \lesssim 
			\frac{|A_n'|^{1/2}}{|A_n|}.
		\]

		If $|\la{A_n}| \geq (1-c_2/2) |A_n|$, the latter estimate can be improved to 
		\begin{equation}\label{eq:improved}
			\|\psi\|_{\la{A_n}} \lesssim \sup_{u\in S} \frac{\|u\|_{\sm{A_n}}}{\|u\|_{A_n}}	 
			\frac{|\sm{A_n}|^{1/2}}{|A_n|}.
		\end{equation}
	\end{lem}
	\begin{proof}
		We first show that $\|\psi\|_{A_n'}$ and $\|\psi\|_{A_n''}$ are both positive.
Write $\psi = \psi_1\charfun_{A_n'} + \psi_2\charfun_{A_n''}$ with $\psi_1,\psi_2\in S$
and assume that $\psi_1=0$ a.s. on $A_n'$. Since $\psi$ is orthogonal to $S = S(\Omega)$ and $\psi_2\in S$, 
		we have 
		\[
			0 = \langle \psi, \psi_2\rangle = \int_{A_n''} |\psi_2|^2\dif\mathbb P,
		\]	
		which implies $\psi_2 =0$ a.s. on $A_n''$ and therefore, $\psi = 0$ a.s. contradicting
		the assumption $\|\psi\|_2 = 1$.  Thus we have $\|\psi_1\|_{A_n'} > 0$.
		By a similar argument, testing $\psi$ with $\psi_1$ instead of $\psi_2$, we obtain
		$\|\psi_2\|_{A_n''} >0$.

		The assertion about $\|\psi\|_{A_n^c}$ is just Lemma~\ref{lem:orth}.
		Observe that the $L^2$-normalization condition of $\psi$ becomes
		\begin{equation}\label{eq:normal}
			1 = \int_{A_n} |\psi|^2\dif\mathbb P = \int_{\la{A_n}}
			|\psi|^2\dif\mathbb P + \int_{\sm{A_n}} |\psi|^2 \dif\mathbb P \gtrsim
			 \|\psi\|_{\la{A_n}}^2 |\la{A_n}| + \|\psi\|_{\sm{A_n}}^2 |\sm{A_n}|,
		\end{equation}
		where in the last inequality we used \eqref{eq:L1Linfty} for 
		the atoms $\la{A_n}$ and $\sm{A_n}$, respectively.
		In particular, this yields the estimates
		\begin{equation}\label{eq:est_start}
			\| \psi \|_{\sm{A_n}} \lesssim |\sm{A_n}|^{-1/2},\qquad 
			\| \psi \|_{\la{A_n}} \lesssim |\la{A_n}|^{-1/2}.
		\end{equation}
		Now we distinguish the following cases.

		\textsc{Case I: } Assume $|\la{A_n}| \leq (1-c_2/2) |A_n|$.
		 In this case we have $|A_n| \simeq |\la{A_n}| \simeq |\sm{A_n}|$.
		 Therefore, the second estimate in \eqref{eq:est_start} implies the assertion.

		 \textsc{Case II: } Assume $|\la{A_n}| \geq (1-c_2/2) |A_n|$.
		By orthogonality of $\psi$ to $S$, we have $\langle \psi,\psi_2\rangle =0$.
		Evaluating this condition yields
		\[
			0 = \int_{A_n'}\psi_1 \overline{\psi_2}\dif\mathbb P	+ \int_{A_n''} |\psi_2|^2\dif\mathbb P,
		\]
		which implies the inequality
		\begin{equation}\label{eq:orth2}
			|A_n''|\|\psi_2\|_{A_n''}^2 \lesssim \int_{A_n''} |\psi_2|^2\dif\mathbb P = 
			\Big| \int_{A_n'} \psi_1\overline{\psi_2}\dif\mathbb P\Big| 
			\leq \| \psi_1\|_{A_n'} \|\psi_2\|_{A_n'} |A_n'|.
		\end{equation}
		Thus, since $\|\psi_2\|_{A_n''} > 0$ by the first part of the proof, we divide by $|A_n''|\|\psi_2\|_{A_n''}$ and use
		the estimate for $\|\psi\|_{\sm{A_n}} = \|\psi_1\|_{\sm{A_n}}$ from \eqref{eq:est_start}
		above to get 
		\[
			\|\psi\|_{A_n''} = \|\psi_2\|_{A_n''} \lesssim \frac{\|\psi_2\|_{A_n'}}{\|\psi_2\|_{A_n''}}
			\frac{|A_n'|^{1/2}}{|A_n''|}.
		\]
		 Since $|A_n''| \geq (1-c_2/2)|A_n|$, inequality \eqref{eq:L1Linfty} implies $\|\psi_2\|_{\la{A_n}} \simeq \|\psi_2\|_{A_n}$.
		 Thus, this inequality and the fact that $\psi_2\in S$ in particular implies \eqref{eq:improved}.
	\end{proof}

\begin{thm}\label{thm:Pnbound_2}
	For every integer $n\geq 0$, every $g\in L^1$
	and every $j\in\{0,\ldots,\ell_n\}$, we have the following pointwise inequality:
	\[
		| P_{n,j} g | \lesssim \mathbb E_{n-1} |g| + \mathbb E_n |g|.
	\]
\end{thm}
\begin{proof} 
	Assume that $n\geq 0$ and $j\in\{0,\ldots,\ell_n\}$. Let $A$ be an atom of $\mathscr F_{n-1}$
	and let $g_1,\ldots, g_M$ be an orthonormal basis of $\{
		f\charfun_{A} : f\in \operatorname{ran} P_{n,j}\}$ so that
		$g_1,\ldots, g_L$ is an orthonormal basis of $\{f\charfun_{A}
		: f\in S_{n-1}\}$. We know that $L\leq M\leq
		2k$. Additionally, we know that $L<M$ only if $A=A_n$ and $j>0$.
		Thus, in order to estimate $P_{n,j} g = \sum_{i=1}^M \langle
		g,g_i\rangle g_i$ on $A$, it suffices to estimate the expression
			$\langle g, g_i\rangle g_i$
		on $A$ for a fixed $1\leq i\leq M$.

		Firstly, assume that $i\leq L$. Then, since $g_{i}$ is
		$L^2$-normalized and contained
		in $S_{n-1}$ and $A$ is an atom of $\mathscr F_{n-1}$, 
		\eqref{eq:L1Linfty} implies
		\[
			1 = \int_A |g_i|^2 \dif\mathbb P \geq d |A| \|g_i\|_{A}^2
		\]
		with $d=c_1^2c_2$. Therefore, we estimate on $A$ as follows
		\begin{equation}\label{eq:whole}
			|\langle g,g_i\rangle g_i| \leq \int_A |g|\dif\mathbb P
			\cdot \|g_i\|_{A}^2 \leq
			\frac{1}{d|A|}\int_A |g|\dif \mathbb P = d^{-1}\cdot \mathbb E_{n-1}|g|.
		\end{equation}

		Secondly, assume that $L<i\leq M$, which implies $A=A_n$. We
		also assume that $n\geq 1$ since if $n=0$, the result follows in
		the same way as in \eqref{eq:whole}. We have to estimate the
		expression $\langle g,g_i\rangle g_i$ both on $\la{A_n}$ and on
		$\sm{A_n}$. Denote $\ell := \|g_i\|_{\la{A_n}}$ and $s :=
		\|g_i\|_{\sm{A_n}}$. 
		Thus, we can estimate $|\langle g,g_i\rangle g_i|$ on $\la{A_n}$ using 
		Lemma~\ref{lem:ess_estimate} as follows:
		\begin{align*}
			|\langle g,g_i\rangle g_i| &\leq \ell^2 \int_{\la{A_n}}
			|g|\dif\mathbb P + \ell s\int_{\sm{A_n}} |g|\dif\mathbb P \\
			&\lesssim \frac{ | \sm{A_n}|}{|\la{A_n}|^2} \int_{\la{A_n}}
			|g|\dif\mathbb P + \frac{1}{|A_n|} \int_{\sm{A_n}} |g|\dif\mathbb P
			\leq \mathbb E_{n-1}|g| + \mathbb E_{n}|g|
		\end{align*}
		and similarly on $\sm{A_n}$
		\begin{align*}
			|\langle g,g_i\rangle g_i| &\leq \ell s \int_{\la{A_n}}
			|g|\dif\mathbb P + s^2\int_{\sm{A_n}} |g|\dif\mathbb P  \\
			&\lesssim \frac{1}{|A_n|} \int_{\la{A_n}}
			|g|\dif\mathbb P +  \frac{1}{|\sm{A_n}|}\int_{\sm{A_n}} |g|\dif\mathbb P
			\leq \mathbb E_{n-1}|g| + \mathbb E_{n}|g|.
		\end{align*}
		This completes the proof of the theorem.
\end{proof}

\begin{example}\label{ex:three}
Here we give an example to show that the assertion of Theorem~\ref{thm:Pnbound_2} is not true in general 
if we split the set $A_n$ into more than two subsets.
We consider the case $\Omega=[0,1]$,
Lebesgue measure $|\cdot|$,  $k=1$ and $S = \lin\{\charfun_{\Omega}\}$.
Let $C_{1}\cup C_{2}\cup C_{3}$ be a partition of $\Omega$ with the following properties:
$|C_{1}| = |C_{2}|$ and $|C_{3}| = \varepsilon$ for some small parameter $\varepsilon>0$.
Define
\[
	h = \alpha_1 \charfun_{C_{1}} - \alpha_2\charfun_{C_{2}} + \alpha_3\charfun_{C_{3}}
\]
where we choose the positive numbers $\alpha_1,\alpha_2\simeq 1$ and 
$\alpha_3\simeq |C_{3}|^{-1/2}=\varepsilon^{-1/2}$ 
such that $\mathbb E h=0$ and $\mathbb E h^2=1$.
Let $g$ be a function with $g=0$ on $C_{3}$ that has the same sign as $h$ on $C_{1}\cup C_{2}$.
Then, we evaluate for $t\in C_{3}$
\[
\langle g, h\rangle h(t) = \int_{C_{1}\cup C_{2}} |g(x)| |h(x)|\dif x \cdot h(t) 
\gtrsim \varepsilon^{-1/2}\int_{C_{1}\cup C_{2}} |g(x)|\dif x.
\]
This implies that an inequality as in Theorem~\ref{thm:Pnbound_2} is not possible as $\varepsilon>0$
was arbitrary.
For more results in this direction, we refer to Section~\ref{sec:more_than_two}.
\end{example}
\subsection{Pointwise and $L^p$-limits of $(P_{n,j}g$)}
In this section we use Theorem~\ref{thm:Pnbound_2} to identify the pointwise a.e. and 
$L^p$-norm limit
of the sequence $(P_{n,j}g)$ for $g\in L^p$ with $1\leq p< \infty$.
As $\cup_n S_n$ may not be dense in $L^p$, this limit need not be the function $g$ itself.

Denote by $\psi_{n,j}$ the unique (up to sign) function contained in
$V_{n,j}$ that is orthonormal to $V_{n,j-1}$.
A simple consequence of Theorem~\ref{thm:Pnbound_2} is that the operators $P_{n,j}$ and
$R_m$ are uniformly bounded on $L^p$ for $1\leq p\leq \infty$ by a constant that depends only on
$c_1,c_2\in (0,1]$ from \eqref{eq:L1Linfty} and $k$.
This implies that $(\psi_{n,j})$ is a basis in  $U_p=\overline{\operatorname{span}}_{L^p} (\cup_n S_n)$ 
for each $1\leq p\leq \infty$, and the operators $P_{n,j}$ are the respective partial sums.

Denote by $P_\infty$ the orthoprojector onto $U_2$. Then, we identify the $L^p$-norm limit of 
$(P_{n,j}g)$ for $g\in L^p$ with $1\leq p<\infty$ as follows.
\begin{prop}
	For $1\leq p<2$, the operator $P_\infty$ extends to a bounded operator on $L^p$
	and for $2\leq p \leq \infty$, the operator $P_\infty$ is bounded on $L^p$.
	In both cases, $P_\infty:L^p \to L^p$ is a projection onto $U_p$.

	Moreover, for all $1\leq p<\infty$ and $g\in L^p$, we have 
	\[
		\| P_{n,j} g - P_\infty g\|_p \to 0.
	\]
\end{prop}
\begin{proof}
	To simplify notation, we assume that the orthonormal basis $(\psi_{n,j})$ consists only 
	of real-valued functions. In the complex-valued case, we need to consider partial sum 
	operators $P_{n,j}$ with respect to the two orthonormal systems $(\psi_{n,j})$ and $(\overline{\psi_{n,j}})$.

	First, assume $1\leq p<2$. In this case, the argument is analogous as in 
	 P.F.X. M\"uller, M. Passenbrunner \cite[Section 3]{MP2020}, but we present it for the sake of the completeness.
	Let $g\in L^p$ and $\varepsilon>0$. Since $L^2$ is dense in 
	$L^p$, we can choose $f\in L^2$ with the property $\|g-f\|_p<\varepsilon$.
	Now, choose $N_0$ sufficiently large so that for all $m,n>N_0$ and $i,j$ we have $\|(P_{n,j} - P_{m,i})f\|_2 < \varepsilon$.
	Then we obtain by the uniform boundedness of $\|P_{n,j} : L^p\to L^p\|$
	\begin{align*}
		\| (P_{n,j} - P_{m,i})g \|_p &\leq \|(P_{n,j}-P_{m,i})(g-f)\|_p + \|(P_{n,j}-P_{m,i})f\|_p	\\
		&\lesssim \varepsilon + \|(P_{n,j}-P_{m,i}) f\|_2 \lesssim \varepsilon.
	\end{align*}
	This implies that $P_{n,j}g$ converges in $L^p$ to some limit $P_{\infty,p}g$. 
	Additionally, $\| P_{\infty,p} : L^p \to L^p\| \leq \sup_{n,j} \| P_{n,j} : L^p\to L^p\|$
	and $P_{\infty,p}$ coincides with $P_\infty$ on $L^2$.
	Next, observe that $P_{\infty,1}$ coincides with $P_{\infty,p}$ on $L^p$ for each $1\leq p \leq2$.
	Therefore, we use the symbol $P_\infty$ to denote $P_{\infty,p}$ for each $1\leq p \leq 2$.
	Observe that $P_\infty:L^p \to L^p$ is a projection onto $U_p$.

Consider the case of $2 < p < \infty$. For $f \in L^p$, $P_\infty f $ is well defined as an element of $L^2$. We need to 
check that $P_\infty:L^p \to L^p$, and moreover it is a projection with the range $U_p$.
Indeed, let $f \in L^p $ and $g \in L^2$. Then, by the boudedness of $P_\infty:L^{p'} \to L^{p'}$ ($p'=p/(p-1)$), 
$$
|\langle P_\infty f, g\rangle| = |\langle f, P_\infty g\rangle | \lesssim \| f \|_p \| P_\infty g\|_{p'}
\leq C \| f \|_p \| g \|_{p'}.
$$
As $L^2$ is dense in $L^{p'}$, this implies that $P_\infty f \in L^p$ and $P_\infty : L^p \to L^p$ is bounded.
The above argument implies also that $P_\infty :L^p \to L^p$ and $P_\infty: L^{p'} \to L^{p'}$ are a pair of Banach space adjoint operators.
Observe that $P_\infty = \operatorname{Id}$ on $\cup_n S_n$ and therefore
$P_\infty = \operatorname{Id}$ on $U_p$. 
Next, we check that $P_\infty f \in U_p$. Clearly, $U_p$ is a convex and norm-closed subset of $L^p$. Therefore, by Mazur's theorem, 
it is enough to see that $P_{n,j} f \to P_\infty f$ weakly in $L^p$ for $f \in L^p$.
 To see this, take  $g \in L^{p'}$.
Using duality and the already proved part   for $1 < p' < 2$
we get
$$
| \langle P_\infty f - P_{n,j} f, g \rangle|
= |\langle f, P_\infty g - P_{n,j} g\rangle| \leq \| f \|_p \| P_\infty g - P_{n,j} g \|_{p'}
\to 0.
$$
Thus, $P_\infty f \in U_p$. 
Consequently, $P_\infty: L^p \to L^p$ is  a bounded projection onto $U_p$.
As $(\psi_{n,j})$ is a basis in $U_p$, it follows that $P_{n,j} f = P_{n,j} (P_\infty f ) \to P_\infty f $ in $L^p$ norm. 
\end{proof}

Using the operator $P_\infty$ on $L^1$, we obtain the following result concerning pointwise 
convergence.
\begin{prop}
	For $g\in L^1$, the sequence $(P_{n,j}g)$ converges $\mathbb P$-a.e to $P_\infty g$,
\end{prop}
\begin{proof}
It suffices to prove the assertion for $g$ contained 
	in the range $U_1$ of $P_\infty$.	But this follows from a standard argument as in \cite[pp. 3--4]{Garsia1970}
	(see also \cite[Theorem 3.2]{MP2020})
	using the following two ingredients:
	\begin{enumerate}
		\item The space $\cup_n S_n$ is dense $U_1$ and
		 for $v\in\cup_n S_n$ we have $P_{n,j} v = v$ for sufficiently large 
			$n$ and therefore in particular $P_{n,j}v \to v$ pointwise.
		\item The maximal operator $P^*f(x) :=\sup_{n,j} |P_{n,j}f(x)|$ is by Theorem~\ref{thm:Pnbound_2}
			bounded by the corresponding martingale maximal function, which, by Doob's maximal inequality,
			is of weak type 1-1. Therefore, the maximal operator $P^*$ is of weak type 1-1 as well. \qedhere
	\end{enumerate}
\end{proof}

\section{Decomposition of $L^1$-functions}\label{sec:decomp}
Let $T$ be a stopping time with respect to the filtration $(\mathscr F_n)$.
Moreover, let $f\in L^1$.
We set $g_n := P_n f$ for
$n\geq -1$.
Then, we define $g_{T} := (g_{T(x)})(x)$ if $T(x) < \infty$ and we set $g_T(x)=
f(x)$ if $T(x)=\infty$.
\begin{prop}\label{prop:stopping}
	Let $T$ be a stopping time and $f\in L^1$. With the above notation, we
	have
	\begin{enumerate}
		\item $\mathbb E |g_T| \lesssim \mathbb E|f|$.
		\item If $n\geq 0$ and  $j\in\{1,\ldots,\ell_n\}$ we have 
			\[
				\charfun_{\{T\geq n\}} (P_{n,j}f)  = \charfun_{\{T\geq
				n\}}(P_{n,j} g_T).
			\]
	\end{enumerate}
\end{prop}
\begin{proof}
	We have by \eqref{eq:commute} and the uniform $L^1$-boundedness of
	the operators $P_{n,j}$
\begin{align*}
	\mathbb E |g_T| &= \int_{\{T=\infty\}} |f| \dif\mathbb P + \sum_n
	\int_{\{T=n\}} |g_n|\dif \mathbb P  \\
	&\lesssim \int_{\{T=\infty\}} |f| \dif\mathbb P + \sum_n \int_{\{T=n\}}
	|f|\dif \mathbb P = \mathbb E|f|,
\end{align*}
which proves item (1). 

	Now we prove item (2):
	since the set $\{T\geq n\}$ is $\mathscr
	F_{n-1}$-measurable and by Remark~\ref{rem:Pm}, we see that
	\begin{align*}
		\charfun_{\{T\geq n\}} (P_{n,j} g_T) &= P_{n,j} ( \charfun_{\{T\geq n\}}
		g_T) = P_{n,j} \Big(\big(\charfun_{\{T=\infty\}} + \sum_{\ell\geq
		n}\charfun_{\{T=\ell\}} \big) g_T\Big) \\
		&= P_{n,j} \Big(
		\charfun_{\{T=\infty\}} f + \sum_{\ell\geq n}
		\charfun_{\{T=\ell\}} P_\ell f \Big)
	\end{align*}
	with the series converging in $L^1$. By \eqref{eq:commute} and the
	continuity of $P_{n,j}$ on $L^1$, we obtain
	\begin{align*}
		\charfun_{\{T\geq n\}} (P_{n,j} g_T) &=  P_{n,j}( \charfun_{\{T=\infty\}}
		f) + \sum_{\ell\geq n} P_{n,j} \big(P_\ell ( \charfun_{\{T=\ell\}} f)
		\big) \\
		&= P_{n,j}( \charfun_{\{T=\infty\}} f) + \sum_{\ell\geq n}
		P_{n,j}(\charfun_{\{T=\ell\}} f),
	\end{align*}
	where the last equation is true since the range of the projector
	$P_{n,j}$ is a subset of the range of the projector $P_\ell$ for $\ell\geq n$.
	We thus obtain
	\[
		\charfun_{\{T\geq n\}} (P_{n,j} g_T) = P_{n,j}\Big( \big(
		\charfun_{\{T=\infty\}} + \sum_{\ell\geq n}
		\charfun_{\{T=\ell\}}\big)f\Big)  = P_{n,j} ( \charfun_{\{T\geq n\}}f)  =
		\charfun_{\{T\geq n\}} (P_{n,j}  f), 
	\]
	since the set $\{T\geq n\}$ is $\mathscr F_{n-1}$-measurable.
\end{proof}

Utilizing the above results---in particular Theorem~\ref{thm:Pnbound_2}---we next show that  an
integrable function $f$ admits a decomposition similar to Gundy's
decomposition for martingales (see e.g. \cite{Pisier2016}).
\begin{thm}\label{thm:gundy_like} 
	Let $f\in S_N$ for some positive integer $N$ with $\mathbb E|f|\leq 1$. 
	Then, for each $\lambda>0$ there exists a
	decomposition $f=a+b+c$ with $a,b,c\in L^1$ satisfying
	\begin{enumerate}
		\item $\| a\|_1 \lesssim 1$ and $\mathbb P( \sup_m
			|da_m|\neq 0) \lesssim \lambda^{-1}$.
		\item $\big\| \sum_m |db_m| \big\|_1 \lesssim 1$.
		\item $\| c\|_1 \lesssim 1$ and $\|c\|_\infty \lesssim \lambda$.
	\end{enumerate}
\end{thm}
\begin{proof}
Let $\lambda>0$ and let $f\in S_N$ with $\mathbb E|f|\leq 1$.
We recall the notation $g_n = P_n f$ and the definition of  the index set
$I_n = \{ m : S_{n-1} \subsetneq V_m \subseteq S_n\}$.
Let $c_3$ be the implicit constant from Theorem~\ref{thm:Pnbound_2}.
Define 
\[
	r = \inf \{ n \geq 0: \mathbb E_n|g_n| > \lambda/(4c_3) \text{ or }
	\max_{m\in I_n} \mathbb E_n | f_{m}| > \lambda\}.
\]
Moreover, define $v_{-1}=0$ and 
\[
	v_n = \sum_{m\in I_n}\charfun_{\{r=n\}} \mathbb E_n( |df_m|) = 
	\sum_{m\in I_n} \mathbb E_n( \charfun_{\{r=n\}} |df_m|),\qquad n\geq 0.
\]
Let 
\[
	s = \inf\{ n\geq -1: \sum_{\ell=-1}^n \mathbb E_\ell v_{\ell+1} +
	\sum_{\ell=0}^n v_\ell > \lambda\}.
\]		
Observe that $r,s$ are both defined to be the smallest value of $n$ such that 
certain inequalities for $\mathscr F_n$-measurable functions are satisfied. Therefore, $r$ and $s$
are both stopping times with respect to the filtration $(\mathscr F_n)$. Consequently
$T:= \min(r,s)$ is also a stopping time with respect to $(\mathscr F_n)$.

\textsc{Part I: }Define
\[
	a := f - g_T.
\]
Then, Proposition~\ref{prop:stopping}  and the assumption $\|f\|_1\leq 1$ immediately imply $\|a\|_1 \lesssim 1$.
If $T(x)=\infty$ we have $da_m(x)=0$ for all $m$ by item (2) of
Proposition~\ref{prop:stopping} and thus
\[
	\{ \sup_m |da_m|\neq 0 \} \subset \{T<\infty\} = \{r<\infty\} \cup
	\{s<\infty\}.
\]
Therefore, by Theorem~\ref{thm:Pnbound_2} and Doob's maximal inequality,
\begin{equation}\label{eq:r_2}
	\mathbb P(r<\infty) \leq  \mathbb P( \sup_{n\geq 0} \max_{m\in I_n} \mathbb
	E_n |f_{m}| > \lambda/(4c_3))
	\leq \mathbb P( 8 c_3^2 \sup_{n\geq 0} \mathbb E_n |f| >\lambda)  \leq
	8c_3^2 \frac{\mathbb E|f|}{\lambda}\lesssim \lambda^{-1}.
\end{equation}
Now we show
that $\mathbb P(s<\infty) \lesssim \lambda^{-1}$.
We begin by using Chebyshev's inequality: 
\begin{equation}\label{eq:start_s}
	\mathbb P(s<\infty) = \mathbb P \Big( \sum_{n=-1}^\infty (\mathbb E_n
	v_{n+1} + v_n) > \lambda \Big) \leq \lambda^{-1} \sum_{n=-1}^\infty \mathbb
	E(\mathbb E_n(v_{n+1})+ v_n) \leq 2\lambda^{-1} \sum_{n=0}^\infty \mathbb Ev_n.
\end{equation}
In order to estimate the integral of $v_n=\sum_{m\in I_n} \mathbb E_n(
|df_m|)\charfun_{\{r=n\}}$, we fix an index $m\in I_n$ and write 
\begin{equation}\label{eq:double1}
	\mathbb E ( |df_m|\charfun_{\{r=n\}} ) \leq \mathbb E( |f_m |\charfun_{\{r=n\}}) + 
	\mathbb E( |f_{m-1}|\charfun_{\{r=n\}}).
\end{equation}

Since for $m\in I_n$, the support of $df_m$ is contained
in $A_n$ by Lemma~\ref{lem:orth}, we know that $\{r = n\} \subset A_n$. If we assume that $\{r=n\}$ is
not empty, this implies by the $\mathscr
F_n$-measurability of the set $\{r=n\}$ that it can
either be $\la{A_n}$, $\sm{A_n}$ or $A_n$ if $n\geq 1$ and $\{ r= 0\}$ can only
be $\Omega$ if it is nonempty.

In order to continue the estimate in \eqref{eq:double1}, 
we are going to show that $\mathbb E(\charfun_{\{r=n\}} |f_m|)\lesssim
\mathbb E(\charfun_{\{r=n\}} |g_n|)$ for $m\in I_n$ or $m$ such that $f_m  =
g_{n-1}$.
To this end, we distinguish two possibilities for the sets $\{r=n\}$:

		\textsc{Case 1: } $\{r=n\} = A_n$ with the atom $A_n$ of $\mathscr F_{n-1}$:
		Since $f_m = P_{n,j} g_n$ for some $j$ and the
		operator $P_{n,j}$ is uniformly bounded on $L^1$,
	\begin{equation}\label{eq:r1}
			\mathbb E_{n-1}( |f_{m}| \charfun_{\{r=n\}} ) =
			\charfun_{\{r=n\}} \mathbb E_{n-1} |f_{ m}| \lesssim
			\charfun_{\{r=n\}} \mathbb E_{n-1} |g_n| = \mathbb E_{n-1}(
			|g_n|\charfun_{\{r=n\}}).
	\end{equation}
	\textsc{Case 2: } Let $A:= \{r=n\} \in \{ \sm{A_n}, \la{A_n} \}$, which means that $n\geq 1$
	and $B := (A_n\setminus A) \subset \{r>n\}$.
		By
		Theorem~\ref{thm:Pnbound_2},
		\begin{equation}\label{eq:ess}
		\begin{aligned}
			\int_{A} |f_{m}| &\leq c_3 \int_{A} \big( \mathbb
			E_{n-1} |g_n| + \mathbb E_n |g_n|\big) =
			c_3 \bigg(\frac{|A|}{|A_n|} \int_{A_n} |g_n| +
			\int_{A} |g_n|\bigg) \\
			& = c_3\bigg(\frac{|A|}{|A|+ |B|} \int_{B} |g_n| + \Big( 1 +
			\frac{|A|}{|A_n|}\Big) \int_{A} |g_n|\bigg) \\
			&\leq c_3\bigg(\frac{|A|}{|A|+ |B|} \int_{B} |g_n| + 
			2 \int_{A} |g_n|\bigg) \leq  
			\frac{\lambda}{4}|A|+ 2c_3 \int_{A} |g_n|,
		\end{aligned}
	\end{equation}
		where in the last inequality, we used the inequality 
		$\int_B |g_n| \leq \lambda|B|/(4c_3)$, which is implied by $B
		\subset \{ r>n\}$.
		Since $A = \{r=n\}$, we know that either $\mathbb E_n |g_n|
		> \lambda/ (4c_3)$ on $A$ or there exists $m_1\in I_n$ so that
		$\mathbb E_n |f_{m_1}| > \lambda$ on $A$. 
		In both cases, using \eqref{eq:ess} for $m=m_1$ in the second case, we
		obtain the inequality
		\begin{equation}\label{eq:clear}
			\int_A |g_n| \gtrsim \lambda|A|.
		\end{equation}
		Using \eqref{eq:clear} we continue the calculation in \eqref{eq:ess} to deduce
		\[
			\int_{A} |f_{m}| \lesssim \int_{A} |g_n|,
		\]
		which implies 
		\begin{equation}\label{eq:r2}
			\mathbb E (\charfun_{\{r=n\}}|f_m| ) =  \mathbb E
			(\charfun_{A} |f_{m}|) \lesssim
			\mathbb E (\charfun_{A} |g_n|) = \mathbb E
			(\charfun_{\{r=n\}} |g_n|).
		\end{equation}

Combining cases 1 and 2, we continue the calculation in \eqref{eq:double1} to
deduce
\[
	\mathbb E( |df_m| \charfun_{ \{r=n\} } ) \lesssim \mathbb E(
	|g_n|\charfun_{\{r=n\}})\lesssim \mathbb E(|f|\charfun_{\{r=n\}}),
\]
where the last inequality follows from $P_n f = g_n$ and the fact that the
operators $P_n$ are uniformly bounded on $L^1$.
Therefore, the definition of $v_n$ implies
\begin{equation}\label{eq:vn}
	\sum_n \mathbb E v_n \lesssim \mathbb E|f|\leq 1.
\end{equation}
Inserting this inequality in \eqref{eq:start_s}, we finally obtain
\[
	\mathbb P(s<\infty) \lesssim \lambda^{-1}.
\]

\textsc{Part II:}
 Observe that 
\[
	g_{T\wedge n} - g_{T\wedge (n-1)} =  \charfun_{\{T\geq n\}}
	(g_n - g_{n-1}) = \charfun_{\{r\geq n\}} \charfun_{\{s\geq n\}}
	(g_n - g_{n-1}).
\]
This gives 
\[
	g_T = \sum_{n\geq 0} (g_{T\wedge n} - g_{T\wedge(n-1)}) = 
	\sum_{n\geq 0}\sum_{m\in I_n} \charfun_{\{r\geq n\}}\charfun_{\{s\geq n\}} df_m.
\]
Now, let $n\geq 0$ and $m\in I_n$ and define
\begin{align*}
	\gamma_m &= \charfun_{\{r > n\}} \charfun_{\{s\geq n\}} df_m,
	\\
	\delta_m &= \charfun_{\{r = n\}} \charfun_{\{s\geq n\}} df_m.
\end{align*}
Recall that by $R_m$ we denote the orthoprojector onto the space $V_m$ so that $df_m =
R_m f - R_{m-1}f$. Then,
observe that since $\{r\geq n\}$ and $\{s\geq n\}$ are $\mathscr{F}_{n-1}$-measurable,
we have $R_{m-1} (\gamma_m + \delta_m)= \charfun_{\{r\geq n\}}
\charfun_{\{s\geq n\}} R_{m-1}(df_m) = 0$ by Remark~\ref{rem:Pm}. This means that $\gamma_m +
\delta_m$ is a difference sequence with respect to the operators $(R_m)$. Thus,
the difference sequences 
\begin{equation*}
 db_m = \delta_m - R_{m-1} \delta_m,\qquad dc_m = \gamma_m + R_{m-1}\delta_m
\end{equation*}
define the functions $b$ and $c$.
Fix the parameter $n$ and consider
\begin{align*}
	\mathbb E \Big[\sum_{m\in I_n} |\delta_m|\Big] &= \mathbb E
\Big[	\sum_{m\in I_n} \charfun_{\{r=n\}} \charfun_{\{s\geq n\}} |df_m| \Big] \\
&\leq \mathbb E \Big[\sum_{m\in I_n} \charfun_{\{r=n\}}
|df_m| \Big]  =\mathbb E v_n.
\end{align*}
Therefore, by estimate \eqref{eq:vn}
\[
	\sum_m \mathbb E|\delta_m| \lesssim \sum_{n} \mathbb E v_n \lesssim 1.
\]
Since the operators $R_m$ are uniformly bounded on $L^1$, this implies $\mathbb E \sum_m |db_m|\lesssim 1$ and also, 
\[
	\| c\|_1 = \Big\| \sum_m dc_m \Big\|_1 \leq \|g_T\|_1 +  \Big\| \sum_m |db_m|\Big\|_1 \lesssim 1.	
\]
It remains to estimate the $L^\infty$-norm of $c$. We have
\[
	\sum_m \gamma_m = g_{ (r-1)\wedge s}.
\]
Let $x\in \{r=n\}$ and for $\ell < n$, let $B$ be the atom of $\mathscr F_\ell$  that
contains $x$.  Since $g_{\ell}\in S_\ell$,
\[
	\|g_{\ell}\|_{B} \lesssim (\mathbb E_{\ell} |g_{\ell}|)(x).
\]
By the minimality of $n$ in the definition of the stopping time $r$, this implies
\[
	\|g_\ell\|_{B} \lesssim \lambda.
\]
This argument yields the inequality
\[
	\Big\| \sum_m \gamma_m \Big\|_{\infty} =\| g_{(r-1)\wedge s} \|_\infty \lesssim \lambda.
\]

Now we estimate the second part of $c$. Since $\{s\geq n\}$ is $\mathscr
F_{n-1}$-measurable, by Remark~\ref{rem:Pm},
\begin{align*}
	\Big|\sum_m R_{m-1}\delta_m\Big| &= \Big| \sum_{n\geq 0} \sum_{m\in I_n}
	R_{m-1}\big(\charfun_{\{s\geq n\}} \charfun_{\{r=n\}} (df_m)\big) \Big| \\
	&= \Big| \sum_{n\geq 0} \charfun_{\{s\geq n\}} \sum_{m\in I_n}
	R_{m-1}\big(\charfun_{\{r=n\}} df_m\big) \Big| \\
	& = \Big| \sum_{n=0}^s \sum_{m\in I_n}
	R_{m-1}\big(\charfun_{\{r=n\}} df_m\big) \Big| \\
	& \leq \Big| \sum_{n=0}^{s-1} \sum_{m\in I_n}
	R_{m-1}\big(\charfun_{\{r=n\}} df_m\big) \Big|  
	+ \Big| \sum_{n= s} \sum_{m\in I_n}
	R_{m-1}\big(\charfun_{\{r=n\}} df_m\big) \Big| \\
	&\leq c_3 \sum_{n=0}^{s-1} \sum_{m\in I_n} (\mathbb E_{n-1} + \mathbb
	E_n)(\charfun_{\{r=n\}} | df_m |) + \sum_{n=s} \sum_{m\in I_n}
	\big|R_{m-1}\big( \charfun_{\{r=n\}} df_m\big)\big|,
\end{align*}
where the last inequality follows from Theorem~\ref{thm:Pnbound_2}. Looking at the
definition of the stopping time $s$, we see that the first term is bounded by
$2c_3 \lambda$. 
For the second term we distinguish the cases where $\{r=n\}=A_n$  is an atom of $\mathscr F_{n-1}$
or $\{r=n\}$ is only an atom of $\mathscr F_{n}$. Let $m\in I_n$.
In the first case, we have by Remark~\ref{rem:Pm} that
\[
	R_{m-1} ( \charfun_{\{r=n\}} df_m) = \charfun_{\{r=n\}} R_{m-1}(df_m)=0.
\]
In the second case, we assume that $A:=\{r=n\} \in \{
	\sm{A_n},\la{A_n}\}$, which yields
that $B:=(A_n\setminus A) \subset \{r>n\}$. Since
\[
	0 = R_{m-1} (\charfun_{A_n} df_m) = R_{m-1}(\charfun_{A} df_m) +
	R_{m-1}(\charfun_{B} df_m),
\]
we get 
\[
	|R_{m-1}( \charfun_{\{r=n\}} df_m)| = |R_{m-1}(\charfun_{B} df_m)|.
\]
The fact that $r>n$ on $B$ now yields
\[
	\| \charfun_{B} df_m \|_\infty \leq \| \charfun_{B} f_m\|_\infty +
	\|\charfun_{B}
	f_{m-1}\|_\infty \lesssim \lambda
\]
by definition of the stopping time $r$ and inequality \eqref{eq:L1Linfty}.
Since the operator $R_{m-1}$ is uniformly bounded
on $L^\infty$, we obtain $\| R_{m-1}(\charfun_{B} df_m) \|_\infty \lesssim
\lambda$. Inserting those estimates into the above, we get that 
\[
	\Big\| \sum_m R_{m-1}\delta_m \Big\|_\infty \lesssim \lambda.
\]
This finishes the proof of the theorem.
\end{proof}

\section{Unconditionality of the differences $df_m$}\label{sec:uncond}
Once we know this Gundy-like decomposition, almost the same proof as in the
martingale case yields the following result. We include it for
completeness.
\begin{thm}\label{thm:uncond}
	Let $f\in L^1$ and let
	$\tilde{f}_m = \sum_{j\leq m} \varepsilon_j df_j$ for some
	sequence of signs $(\varepsilon_j)$. Then
	\begin{equation}\label{eq:wt}
		\sup_{\lambda>0} \lambda \mathbb P(\sup_m |\tilde{f}_m|
		>\lambda) \lesssim \mathbb E|f|.
	\end{equation}
\end{thm}
\begin{proof}
It suffices to prove the result for functions in $\cup_n S_n$.
Indeed, assuming the inequality for functions in $\cup_n S_n$ and $f\in L^1$, we 
obtain 
\[
\mathbb P( \sup_m |\tilde{f}_m|>\lambda) = \lim_n \mathbb P (\sup_{m\leq n} |\tilde{f}_m|>\lambda) 	
\lesssim \liminf_n \frac{\|f_n\|_1}{\lambda} \lesssim \frac{\|f\|_1}{\lambda},
\]
where we used \eqref{eq:wt} for functions in $S_n$ and the uniform boundedness of 
the operators $(P_n)$ on $L^1$, respectively.

	Therefore, we now prove the stated inequality 
	for functions $f\in\cup_n S_n$  using Theorem~\ref{thm:gundy_like}. Assume
	w.l.o.g. that $\mathbb E|f|\leq 1$. Then we decompose the function $f$ by
	using Theorem~\ref{thm:gundy_like} into $f=a+b+c$. Then, we see that
	$\tilde{f}_m = \tilde{a}_m + \tilde{b}_m + \tilde{c}_m$ and we have
	\[
		\mathbb P( \sup_m |\tilde{f}_m| >3\lambda) \leq 
		\mathbb P( \sup_m |\tilde{a}_m| >\lambda) + \mathbb P( \sup_m
		|\tilde{b}_m| >\lambda) + \mathbb P( \sup_m |\tilde{c}_m| >\lambda)
	\]
	and we estimate each of those three parts separately.
	\begin{enumerate}
		\item If $\sup_m |\tilde{a}_m| >\lambda$ this implies in
			particular that $\sup_m |da_m|\neq 0$ and therefore
			\[
				\mathbb P(\sup_m |\tilde{a}_m| > \lambda) \leq
				\mathbb P(\sup_m |da_m|\neq 0) \lesssim
				\lambda^{-1}
			\]
			by item (1) of Theorem~\ref{thm:gundy_like}.
		\item For the second term, we estimate
			\[
				\sup_m |\tilde{b}_m| \leq \sum_m |db_m|,
			\]
			which by Chebyshev's inequality  implies 
			\[
				\mathbb P(\sup_m |\tilde{b}_m| > \lambda) \leq
				\lambda^{-1} \mathbb E \sum_m |db_m| \lesssim
				\lambda^{-1},
			\]
			where the latter inequality is a consequence of item (2)
			in Theorem~\ref{thm:gundy_like}.
		\item First we note that $\tilde{c}_m = R_m \tilde{c}_N$.
			Therefore, by Theorem~\ref{thm:Pnbound_2} and Doob's
			inequality,
			\[
				\| \sup_m |\tilde{c}_m| \|_2 \leq 2c_3 \|
				\tilde{c}_N\|_2.
			\]
			Since the terms $dc_j$ are orthogonal to each other for
			different indices $j$, we obtain
			\[
				\| \tilde{c}_N \|_2 = \|c_N\|_2,
			\]
			which by Chebyshev's inequality yields
			\[
				\mathbb P(\sup_m |\tilde{c}_m| >\lambda)\leq
				\lambda^{-2} \mathbb E \sup_m |\tilde{c}_m|^2
				\lesssim \lambda^{-2}
				\mathbb E |c_N|^2 \leq \lambda^{-2}
				\|c_N\|_\infty \mathbb E|c_N|.
			\]
			Using now item (3) of Theorem~\ref{thm:gundy_like} yields
			\[
				\mathbb P(\sup_m |\tilde{c}_m| > \lambda)
				\lesssim \lambda^{-1}.
			\]
	\end{enumerate}
	Combining those steps yields the assertion of the theorem.
\end{proof}

Now it is a simple task to prove the unconditionality 
in $L^p$, $1<p<\infty$ for the differences $(df_m)$.
\begin{proof}[Proof of Theorem~\ref{thm:uncond_intro}]
Just use Theorem~\ref{thm:uncond} and 
the Marcinkiewicz interpolation theorem.
\end{proof}

\section{Democracy of the differences $df_m$}\label{sec:democracy}
We denote by $\psi_{n,j}$ the unique (up to sign) function contained in
$V_{n,j}$ that is orthonormal to $V_{n,j-1}$.
Recall the pointwise estimate from Lemma~\ref{lem:ess_estimate} which reads
\begin{equation}\label{eq:est_psi_2}
		|\psi_{n,j}| \lesssim \begin{cases}
			0,& \text{on } A_n^c, \\
			|\sm{A_n}|^{-1/2},& \text{on } \sm{A_n}, \\
			|\sm{A_n}|^{1/2}/|\la{A_n}|,& \text{on } \la{A_n}.
		\end{cases}
	\end{equation}
	We use this estimate to perform the following calculation:
	\begin{equation}\label{eq:Aprime0}
	\begin{aligned}
		1 = \|\psi_{n,j}\|_2^2 &= \int_{\sm{A_n}}
		|\psi_{n,j}|^2\dif\mathbb P + \int_{\la{A_n}} |\psi_{n,j}|^2\dif\mathbb P
		\leq |\sm{A_n}| \|\psi_{n,j}\|_{\sm{A_n}}^2 +  |\la{A_n}|
		\|\psi_{n,j}\|_{\la{A_n}}^2 \\
		&\lesssim   |\sm{A_n}| \|\psi_{n,j}\|_{\sm{A_n}}^2 +
		\frac{|\sm{A_n}|}{|\la{A_n}|}, 
	\end{aligned}
\end{equation}
which yields that there exists a constant $d>0$ so that if $|\sm{A_n}| /
|\la{A_n}| \leq d$, we get the inequality
\begin{equation}\label{eq:lower_estimate}
	\| \psi_{n,j} \|_{\sm{A_n}} \simeq |\sm{A_n}|^{-1/2}.
\end{equation}

For $1<p<\infty$, Theorem~\ref{thm:uncond_intro} implies that 
\begin{equation}\label{eq:uncondLp}
	\big\| \sum_{n,j} a_{n,j}\psi_{n,j} \big\|_p \simeq \Big\| \big( \sum_{n,j}
	|a_{n,j}\psi_{n,j}|^2 \big)^{1/2} \Big\|_p
\end{equation}
with constants that also depend on $p$.

Next, we want to check that we can also use some variant of the square function.
Before we formulate this result we recall Stein's martingale inequality \cite{Stein1970a}, which
reads for $1<p<\infty$ as
\begin{equation}\label{eq:steins_inequality}
	\Big\| \big(\sum_n |\mathbb E_{n} h_n|^2\big)^{1/2} \Big\|_p \lesssim 
	\Big\| \big(\sum_n |h_n|^2\big)^{1/2} \Big\|_p,
\end{equation}
where $h_n$ are arbitrary functions and the implied constant only depends
on $p$. 
\begin{lem}\label{lem:diff_square}
	For all $1<p<\infty$, we have
	\[
		\Big\| \big( \sum_{n,j} |a_{n,j}\psi_{n,j}|^2 \big)^{1/2} \Big\|_p
		\simeq
		\Big\| \big( \sum_{n,j} |a_{n,j}|^2
		\frac{\charfun_{\sm{A_n}}}{|\sm{A_n}|} \big)^{1/2} \Big\|_p,
	\]
	where the implied constants depend on $c_1,c_2\in(0,1]$ given by \eqref{eq:L1Linfty} and on $k$ and $p$.
\end{lem}
\begin{proof}
	Note that estimate \eqref{eq:est_psi_2} implies also that 
	\[
		|\psi_{n,j}|\lesssim \mathbb E_{n-1}
		\Big(\frac{\charfun_{\sm{A_n}}}{|\sm{A_n}|^{1/2}}\Big)\qquad
		\text{on } \la{A_n}.
	\]
	Therefore, by Stein's inequality \eqref{eq:steins_inequality}
	\begin{equation}\label{eq:lo}
		\Big\| \big(\sum_{n,j} |a_{n,j}\psi_{n,j}|^2\big)^{1/2} \Big\|_p \lesssim 
		\Big\| \big( \sum_{n,j} |a_{n,j}|^2
		\frac{\charfun_{\sm{A_n}}}{|\sm{A_n}|} \big)^{1/2} \Big\|_p.
	\end{equation}

	To get the converse inequality, first fix $n$ and consider the two cases

	\textsc{Case 1: } $|\sm{A_n}| / |\la{A_n}| \leq d$. Here we estimate for
	all $j$, using \eqref{eq:lower_estimate} and \eqref{eq:L1Linfty} 
	\[
		\frac{\charfun_{\sm{A_n}}}{ |\sm{A_n}|^{1/2} } \lesssim \mathbb
		E_n|\psi_{n,j}|.
	\]

	\textsc{Case 2: } $|\sm{A_n}|/|\la{A_n}| > d$. Here, by
	\eqref{eq:Aprime0}, we have
	\[
		\|\psi_{n,j}\|_{\sm{A_n}} \gtrsim
		|\sm{A_n}|^{-1/2}\qquad\text{or}\qquad \|\psi_{n,j}\|_{\la{A_n}}
		\gtrsim |\la{A_n}|^{-1/2} \geq d^{-1/2} |\sm{A_n}|^{-1/2}.
	\]
	Using again \eqref{eq:L1Linfty}, this implies
	\[
		\frac{\charfun_{\sm{A_n}}}{|\sm{A_n}|^{1/2}}\lesssim \mathbb
		E_{n-1}|\psi_{n,j}|.
	\]

	Combining those two cases, we arrive at the inequality
	\[
		\frac{\charfun_{\sm{A_n}}}{|\sm{A_n}|^{1/2}}\lesssim \mathbb
		E_{n-1} |\psi_{n,j}| + \mathbb E_n |\psi_{n,j}|.
	\]
	Using Stein's inequality \eqref{eq:steins_inequality}, we conclude
	\[
		\Big\| \big( \sum_{n,j} |a_{n,j}|^2
		\frac{\charfun_{\sm{A_n}}}{|\sm{A_n}|} \big)^{1/2} \Big\|_p
		\lesssim \Big\| \big(\sum_{n,j} |a_{n,j}\psi_{n,j}|^2\big)^{1/2}
		\Big\|_p, 
	\]
	which, together with \eqref{eq:lo} completes the proof of the lemma.
\end{proof}
For $1<p<\infty$,  let $\psi_{n,j}^p = \psi_{n,j} / \|\psi_{n,j}\|_p$ be the
renormalized version of the functions $\psi_{n,j}$. Note that 
using the above lemma, we obtain
\begin{equation}\label{eq:pnorm}
	\|\psi_{n,j}\|_{p}\simeq |\sm{A_n}|^{1/p - 1/2}.
\end{equation}
Then, we prove the following theorem.
\begin{thm}\label{thm:temlyakov-property}
	For $1<p<\infty$, the functions $(\psi_{n,j}^p)$ are democratic in $L^p$, i.e., we have for all
	finite sets $\Lambda$ of indices $(n,j)$ the
	equivalence
	\[
		\big\|\sum_{(n,j)\in \Lambda} \psi_{n,j}^p \big\|_p \simeq
		(\card\Lambda)^{1/p},
	\]
	where $\card\Lambda$ denotes the cardinality of $\Lambda$
	and 
	the implied constants depend on $c_1,c_2\in(0,1]$ given by \eqref{eq:L1Linfty} and on $k$ and $p$.
\end{thm}
\begin{proof}
	First, we remark that Lemma~\ref{lem:diff_square} and \eqref{eq:pnorm}
	imply
	\[
		\Big\| \big( \sum_{n,j} |a_{n,j}\psi_{n,j}^p|^2 \big)^{1/2} \Big\|_p
		\simeq
		\Big\| \big( \sum_{n,j} |a_{n,j}|^2
		\frac{\charfun_{\sm{A_n}}}{|\sm{A_n}|^{2/p}} \big)^{1/2}
		\Big\|_p.
	\]	
	Therefore, using also~\eqref{eq:uncondLp}, we obtain
	\begin{equation}\label{eq:psi}
		\big\|\sum_{(n,j)\in \Lambda} \psi_{n,j}^p \big\|_p^p \simeq 
		\Big\| \big( \sum_{(n,j)\in\Lambda} 
		\frac{\charfun_{\sm{A_n}}}{|\sm{A_n}|^{2/p}} \big)^{1/2}
		\Big\|_p^p = \int_{\Omega} \Big(\sum_{(n,j)\in\Lambda} 
		\frac{\charfun_{\sm{A_n}}}{|\sm{A_n}|^{2/p}}
		\Big)^{p/2}\dif\mathbb P.
	\end{equation}
	Note that for different indices $n,m$ we have either $\sm{A_n} \cap
	\sm{A_m} =\emptyset$ or one is strictly included in the other. Moreover,
	if $\sm{A_m}\subsetneq \sm{A_n}$ we have geometric decay of measures
	$|\sm{A_m}| \leq |\sm{A_n}|/2$. Therefore, by a geometric series argument
	and the remark that for each $n$, there are at most $k$ numbers $j$ so
	that $(n,j)\in\Lambda$, we get pointwise
	\[
		\Big(\sum_{(n,j)\in\Lambda} \frac{\charfun_{\sm{A_n}}}{| \sm{A_n}|^{2/p}} \Big)^{p/2}	
		\simeq \Big( \max_{(n,j)\in\Lambda} \frac{\charfun_{\sm{A_n}}}{|
			\sm{A_n}|^{2/p}} \Big)^{p/2} = 
		 \max_{(n,j)\in\Lambda} \frac{\charfun_{\sm{A_n}}}{| \sm{A_n}|}
		 \simeq \sum_{(n,j)\in\Lambda} \frac{\charfun_{\sm{A_n}}}{|
			 \sm{A_n}|}.
	\]
	Inserting this in \eqref{eq:psi}, we obtain
	\[
		\big\|\sum_{(n,j)\in \Lambda} \psi_{n,j}^p \big\|_p^p \simeq
		\int_{\Omega}  \sum_{(n,j)\in\Lambda}
		\frac{\charfun_{\sm{A_n}}}{| \sm{A_n}|} \dif\mathbb P =
		\card\Lambda,
	\]
	which concludes the proof.
\end{proof}

If we use this theorem, together with Theorem~\ref{thm:uncond_intro} and
Theorem~\ref{thm:dense_suff} we get
\begin{cor}
For $1<p<\infty$, the functions $(\psi_{n,j}^p)$ are a greedy basis in the $L^p$-closure 
of $\cup_n S_n$.

If additionally $\big| \bigcap_{f\in S} \{ f = 0\} \big| = 0$ and $\cup_n\mathscr F_n$ generates 
$\mathscr F$, the functions $(\psi_{n,j}^p)$ form a greedy basis in $L^p$.
\end{cor}

\section{Infinite measure spaces}\label{sec:infty}
Let $(\Omega,\mathscr G,\mathfrak m)$ be a measure space with $\mathfrak m(\Omega) = \infty$.
Let $(\mathscr G_n)_{n\in\mathbb Z}$ be a filtration on $\Omega$ such that 
for each $n\in\mathbb Z$, the $\sigma$-algebra $\mathscr G_n$ is generated by 
$\mathscr G_{n-1}$ and the subdivision of \emph{each} atom $A$ of $\mathscr G_{n-1}$
into two atoms $A', A''$ of $\mathscr G_n$ having positive and finite measure.

Let $\mathscr A_n$ be the collection of all atoms of $\mathscr G_n$ and set 
$\mathscr A := \cup_{n\in\mathbb Z} \mathscr A_n$. 
Let $S$ be a finite-dimensional space of $\mathscr G$-measurable functions on $\Omega$
such that there exist constants $c_1,c_2\in (0,1]$ so that
for each $f\in S$ and each $A\in\mathscr A$ we have 
	\begin{equation*}
			|\{ \omega\in A : |f(\omega)| \geq c_1\|f\|_{A}
		\}| \geq c_2|A|.
	\end{equation*}

For each atom $A\in\mathscr A$, consider the spaces $V_{A,m}$ satisfying 
\[
S(A) = V_{A,0} \subseteq \cdots \subseteq V_{A,{\ell_A}}	= S(A') \oplus S(A'')
\]
and $\dim V_{A,m} / V_{A,m-1} = 1$ for each $m = 1,\ldots,\ell_A$. Moreover, consider 
the functions $\psi_{A,m}\in V_{A,m}$ satisfying $\psi_{A,m} \perp V_{A,m-1}$ and $\|\psi_{A,m}\|_2 = 1$.

We now show that for each $1<p<\infty$, the renormalized system 
$\{ \psi_{A,m}^p : A\in \mathscr A, 1\leq m\leq \ell_A\}$, given by $\psi_{A,m}^p = 
\psi_{A,m}/\|\psi_{A,m}\|_p$, is unconditional and democratic in $L^p$.
We achieve this by using our results for probability spaces locally on atoms.

To this end, consider finitely many pairs $\Lambda$ of indices $(A,m)$ with 
$A\in\mathscr A$ and $1\leq m\leq \ell_A$. Choose $n_0\in \mathbb Z$ sufficiently small 
such that for each $(A,m)\in \Lambda$, there exists an atom $B\in \mathscr A_{n_0}$ such that 
$A\subseteq B$.
Let 
\[
\Gamma = \{ B\in\mathscr A_{n_0}\;|\; \text{there exists } (A,m)\in\Lambda \text{ with } A\subseteq B \}.
\]
For each $B\in\Gamma$, we consider the probability space $(B,\mathscr G\cap B, \mathfrak m(\cdot) / \mathfrak m(B))$
and a binary filtration $(\mathscr F_n^B)_{n\geq 0}$ on this probability space 
that contains the $\sigma$-algebras $(\mathscr G_n \cap B)_{n\geq n_0}$.
Moreover, we choose the spaces $(V_{n,j})$
(as given in Section~\ref{sec:intro})
corresponding to $(\mathscr F_n^B)_{n\geq 0}$ such that 
\begin{equation}\label{eq:spaces}
\{ V_{n,j} : n\geq 0, 0\leq j\leq \ell_n\} = \{ V_{A,m} : A\in\mathscr A, A\subseteq B, 0\leq m\leq \ell_A\}.
\end{equation}
Define $(\psi_{n,j})$ corresponding to the spaces $(V_{n,j})$ as in the beginning of Section~\ref{sec:democracy}
and consider, for $1<p<\infty$, the renormalized version $\psi_{n,j}^p = \psi_{n,j}/\|\psi_{n,j}\|_p$.
Because of equality \eqref{eq:spaces}, the corresponding functions $\psi_{n,j}$ and $\psi_{A,m}$
also coincide and we deduce from the results in Sections~\ref{sec:uncond} and \ref{sec:democracy}
that for each fixed $B\in\Gamma$, the system 
\[
	\{ \psi_{A,m}^p : A\in\mathscr A, A\subset B, 1\leq m\leq \ell_A\}
\]
is unconditional and democratic in $L^p$ with the same constants for each $B\in\Gamma$.
Let $ f = \sum_{(A,m)\in\Lambda} a_{A,m} \psi_{A,m}^p$ for some coefficients $(a_{A,m})$.
Then we estimate
\begin{align*}
\Big\| \sum_{(A,m)\in\Lambda} \pm a_{A,m} \psi_{A,m}^p \Big\|_p^p &= 
\Big\| \sum_{B\in\Gamma} \sum_{(A,m)\in\Lambda : A\subseteq B} \pm a_{A,m} \psi_{A,m}^p \Big\|_p^p  \\
&= 
 \sum_{B\in\Gamma}\Big\|  \sum_{(A,m)\in\Lambda : A\subseteq B} \pm a_{A,m} \psi_{A,m}^p \Big\|_p^p ,
\end{align*}
which is true since the sets $B\in\Gamma$ are disjoint and
 the support of $\psi_{A,m}^p$ is a subset of $A$.
 By the unconditionality of $\{\psi_{A,m}^p : A\subset B,1\leq m\leq \ell_A\}$ for 
 each $B\in \Gamma$, we obtain 
 \begin{align*}
 \sum_{B\in\Gamma}\Big\|  \sum_{(A,m)\in\Lambda : A\subseteq B} \pm a_{A,m} \psi_{A,m}^p \Big\|_p^p 
 \simeq 
 \sum_{B\in\Gamma}\|  f\charfun_B \|_p^p.
 \end{align*}
 The latter term is now equal to $ \|f\|_p^p$,
 which shows the unconditionality of $(\psi_{A,m}^p)$ in $L^p$.
 Similarly, by the uniform democracy of  $\{\psi_{A,m}^p : A\subset B,1\leq m\leq \ell_A\}$ for 
 each $B\in \Gamma$,
\begin{align*}
\Big\| \sum_{(A,m)\in\Lambda} \psi_{A,m}^p \Big\|_p^p &= 
\Big\| \sum_{B\in\Gamma} \sum_{(A,m)\in\Lambda : A\subseteq B}  \psi_{A,m}^p \Big\|_p^p  
=\sum_{B\in\Gamma}\Big\|  \sum_{(A,m)\in\Lambda : A\subseteq B} \psi_{A,m}^p \Big\|_p^p \\
&\simeq \sum_{B\in\Gamma} \card \{ (A,m)\in \Lambda : A\subseteq B\} = \card \Lambda,
\end{align*}
which also shows the democracy of $(\psi_{A,m}^p)$ in $L^p$.

\section{Non-binary filtrations}
\label{sec:more_than_two}

In this last section, we expand on the setting when the filtration 
$(\mathscr F_n)$ on the probability space $(\Omega,\mathscr F,\mathbb P)$
 is not necessarily binary, meaning that an 
atom $A_n$ of $\mathscr F_{n-1}$ is divided into $r = r(n)\geq 2$ atoms 
$A_{n,1}, \ldots, A_{n,r}$ of $\mathscr F_n$.
We assume that for all $n$, the ordering of those atoms is such that $|A_{n,1}| \leq \cdots \leq |A_{n,r}|$
and that $r(n)$ is bounded independently of $n$.

Define the intermediate $\sigma$-algebras $\mathscr F_{n-1} = \mathscr F_{n,1}  \subseteq 
\cdots \subseteq  \mathscr F_{n,r} = \mathscr F_n$ by the condition that for $\mu = 1,\ldots,r$ 
and restricted to $A_{n}$, the $\sigma$-algebra $\mathscr F_{n,\mu}$ consists of the atoms 
\[
	A_{n,1},\ldots, A_{n,\mu-1}, A_{n,\mu}\cup \cdots\cup A_{n,r}.
\]
Create a binary filtration $(\mathscr F_n')$ that consists precisely of the $\sigma$-algebras 
$\mathscr F_{n,\mu}$ with $n$ being a positive integer and $1\leq \mu\leq r(n)$.
Denote by $\mathscr A, \mathscr A'$ the collection of atoms of the filtrations $(\mathscr F_n)$
and $(\mathscr F_n')$ respectively.

We now show that if a space $S$ satisfies \eqref{eq:L1Linfty} for the filtration $(\mathscr F_n)$,
it also satisfies \eqref{eq:L1Linfty} for the binary filtration $(\mathscr F_n')$.
It is essential here that the refinement of $\mathscr F_{n-1}$ to $\mathscr F_n$ is done such that we 
take away the smallest remaining atom in each step.
\begin{prop}
		Let $S\subseteq L^\infty$ be a finite-dimensional vector space so that 
		there exist constants $c_1,c_2\in (0,1]$ such that  
		for each atom $A\in \mathscr A$ we have \eqref{eq:L1Linfty}, i.e.,
		\begin{equation}\label{eq:L1Linftynew}
				|\{ \omega\in A : |f(\omega)| \geq c_1\|f\|_{A}
			\}| \geq c_2|A|,\qquad f\in S.
		\end{equation}

		Then, there exists a constant $c_2'\in (0,1]$,
		depending only on $c_2$ and $\sup_n r(n)$,
		such that for all $A\in\mathscr A'$
		we have 
		\begin{equation}\label{eq:L1Linfty_otherparameters}
				|\{ \omega\in A : |f(\omega)| \geq c_1\|f\|_{A}
			\}| \geq c_2'|A|,\qquad f\in S.
		\end{equation}
\end{prop}
\begin{proof}
Fix $n$ and assume \eqref{eq:L1Linftynew} for $A\in \{A_n,A_{n,1},\ldots,A_{n,r}\}$.
Let $d_1 = c_2$ and $d_2 = d_1^2/(1+d_1)=(1-\alpha) d_1$ with  $\alpha=(1+d_1)^{-1}<1$.
We show that the set  $A' := A_{n,2}\cup\cdots\cup A_{n,r}$ satisfies \eqref{eq:L1Linfty_otherparameters}
with the constant $c_2' = d_2$.
To this end, we distinguish two cases:

	\textsc{Case 1: } $|A_{n,1}|\leq  \frac{d_1}{1+d_1} | A_n| = \alpha d_1| A_n|$:
			here we estimate for $f\in S$
			\begin{align*}
				|\{ \omega\in A' :
				|f(\omega)|\geq c_1 
				\| f\|_{A' } \}| 
				 &\geq 
					|\{ \omega\in A' : |f(\omega)|\geq c_1 
				\| f\|_{ A_n } \}|  \\
				&\geq (1-\alpha)d_1  | A_n| \geq (1-\alpha)d_1|A'| = d_2|A'|,
			\end{align*}
			where we used \eqref{eq:L1Linftynew} for $A =  A_n$.

		\textsc{Case 2: } $|A_{n,1}| \geq \alpha d_1|A_n|$:
		let $\mu\in \{2,\ldots,r\}$ be such that $\| f\|_{A'} = \|f\|_{A_{n,\mu}}$. 
		Then we estimate
		\begin{align*}
			|\{ \omega\in A' :
			|f(\omega)|&\geq c_1 
				\| f\|_{A' } \}| 
				 \geq |\{ \omega\in A_{n,\mu} : |f(\omega)|\geq c_1\|f\|_{A_{n,\mu}} \}| 
				 \geq d_1|A_{n,\mu}|,
		\end{align*}
		where the last inequality follows from \eqref{eq:L1Linftynew} with $A=A_{n,\mu}$. Continue 
		with $d_1 |A_{n,\mu}| \geq d_1 |A_{n,1}| \geq \alpha d_1^2 | A_{n}| \geq  \alpha d_1^2 |A'| = d_2|A'|$.

Summarizing the two cases, we get inequality \eqref{eq:L1Linfty_otherparameters} for $A = A'$ with the 
parameter $c_2'=d_2$.
Defining the numbers $d_{j+1} = d_{j}^2/(1+d_j)$ recursively, we perform this argument inductively to 
obtain that all atoms $A'\subseteq A_n$ in $\mathscr A'$
satisfy inequality \eqref{eq:L1Linfty_otherparameters} with the parameter $c_2'=d_{r-1}\in (0,1]$.
\end{proof}
\begin{example}
	Consider the space $\Omega= [0,1]$ and $S = \lin \{ p\}$ with 
	$p(t) = t$ for $t\in [0,1]$. Divide $[0,1]$ into the partition of the 
	intervals $[0,\delta], (\delta, 1-\varepsilon), [1-\varepsilon,1]$.
	Then, Section~\ref{sec:examples} shows that \eqref{eq:L1Linfty} is satisfied 
	in this case for the sets $\{ [0,1], [0,\delta], (\delta, 1-\varepsilon), [1-\varepsilon,1]\}$
	with some uniform constants $c_1,c_2\in(0,1]$ in $\delta,\varepsilon$.

	On the other hand, it is also clear that for each choice of constants
	 $c_1',c_2'\in(0,1]$ there are parameters $\delta,\varepsilon>0$ such that the converse 
	 of \eqref{eq:L1Linfty} is satisfied for the set $K = [0,\delta]\cup [1-\varepsilon,1]$ and 
	 the function  $f= p_1$, i.e., we have 
	 \[
		| \{ t\in K : p_1(t) \geq c_1' \}| < c_2' | K | = c_2' (\delta + \varepsilon).
	 \]
	 Thus, if the splitting is not chosen in the way described above, for the intermediate 
	 atoms we  do not necessarily have \eqref{eq:L1Linfty} uniformly for all 
	 partitions.
\end{example}

Let $S\subseteq L^\infty$ be a finite-dimensional vector space that satisfies \eqref{eq:L1Linftynew}
for all $A\in\mathscr A$. The above proposition shows that \eqref{eq:L1Linfty_otherparameters}
is satisfied for all $A\in\mathscr A'$ and some different constant $c_2'\in (0,1]$.
Since $(\mathscr F_n')$ is a binary filtration, we get from our previous results that the 
corresponding orthonormal system  $(\psi_{n,j}')$ (defined in 
the beginning of Section~\ref{sec:democracy}) is unconditional in $L^p$ and 
its $p$-renormalization is democratic in $L^p$ for $1<p<\infty$.
Therefore, for $1<p<\infty$ and the orthonormal system $(\psi_{n,j}')$, 
we have the estimate
\begin{equation}\label{eq:uncond_plus_est}
\Big\| \sum_{n,j}\pm \langle f,\psi_{n,j}' \rangle \psi_{n,j}' \Big\|_p \leq C\|f\|_p,\qquad f\in L^p
\end{equation}
for some positive  constant $C$ independent of $f$ and independent of the choice of signs.
For each $n$ and $\mu = 1,\ldots,r-1$, let $(\psi_{n,\mu,\ell})_\ell$ consist of those finitely many 
functions $\psi_{n,j}'$ satisfying $\psi_{n,j}' \in S(\mathscr F_{n,\mu+1})$ and 
$\psi_{n,j}' \perp S(\mathscr F_{n,\mu})$.
Inequality \eqref{eq:est_psi_2} implies, for each $\mu=1,\ldots, r-1$ and each  $\ell$,
\begin{equation}\label{eq:est_psi_more_than_two}
		|\psi_{n,\mu,\ell}| \lesssim \begin{cases}
			0,& \text{on } (A_{n,\mu}\cup \cdots \cup A_{n,r})^c, \\
			|A_{n,\mu}|^{-1/2},& \text{on } A_{n,\mu}, \\
			|A_{n,\mu}|^{1/2}/|A_n|,& \text{on } A_{n,\mu+1}\cup \cdots \cup A_{n,r}.
		\end{cases}
	\end{equation}

On the other hand, we choose, for each $n$, an $L^2$-normalized function $\psi_{A_n}$ that 
is contained in $S(\mathscr F_n)$ and orthogonal to $S(\mathscr F_{n-1})$, meaning that 
\begin{equation}\label{eq:form_psi}
	\psi_{A_n} = \sum_{\mu=1}^{r-1} \sum_\ell  a_{n,\mu,\ell} \psi_{n,\mu,\ell}
\end{equation}
for some scalars $(a_{n,\mu,\ell})$ with $\sum_{\mu=1}^{r-1}\sum_\ell |a_{n,\mu,\ell}|^2=1$.
The number of terms in the inner sum over $\ell$ is bounded from above 
by the dimension of the space~$S$.

Having set up the notation, we now investigate the question whether the functions 
$(\psi_{A_n})_n$ can also be unconditional/democratic in $L^p$, $1<p<\infty$.
As we will see below, this is not possible in general and this fact extends 
Example~\ref{ex:three}, where we showed that certain pointwise inequalities 
for $\psi_{A_n}$ are not possible in general.

Since $(\psi_{n,j}')$ is unconditional in $L^p$ for $1<p<\infty$, Lemma~\ref{lem:diff_square}
implies
\begin{equation}
	\label{eq:psiequiv}
\begin{aligned}
	\|\psi_{A_n}\|_p &\simeq 	 \Big\| \Big( \sum_{\mu=1}^{r-1} 
	\sum_\ell |a_{n,\mu,\ell}|^2 \frac{\charfun_{A_{n,\mu}}}{|A_{n,\mu}|} \Big)^{1/2} \Big\|_p
	= \Big( \sum_{\mu=1}^{r-1} |A_{n,\mu}|^{1 - p/2} \big(\sum_\ell |a_{n,\mu,\ell}|^2 \big)^{p/2} \Big)^{1/p}\\
	&\simeq \sum_{\mu=1}^{r-1} |A_{n,\mu}|^{1/p - 1/2} \sum_\ell |a_{n,\mu,\ell}|.
\end{aligned}
\end{equation}
In the analysis of the functions $(\psi_{A_n})$, the latter expressions are important, so 
we make the following definition.
\begin{defin}
	Let $1\leq p\leq \infty$. We say that the system $(\psi_{A_n})$ satisfies condition 
	$(O_p)$, if there exists a positive  constant $C$ such that for all $n$ and all $\mu,\nu=1,\ldots,r-1$,
	\[
			\sum_{\ell} |a_{n,\mu,\ell}|	\sum_{m} |a_{n,\nu,m}| |A_{n,\mu}|^{1/p-1/2} 
			|A_{n,\nu}|^{1/2-1/p} \leq C.
	\]	
\end{defin}

We note that condition $(O_2)$ is satisfied for any choice of functions $(\psi_{A_n})$.
 Moreover, $(O_p)$ is satisfied if and only if $(O_{p'})$ is satisfied with 
$p' = p/(p-1)$. Additionally, if $(O_p)$ is satisfied for some $p\in[1,2)$,
then we also have $(O_q)$ for all $q \in (p,2]$.

Note that condition $(O_p)$ is satisfied in particular if  the measures of the atoms $A_{n,1},\ldots,A_{n,r-1}$
are comparable with a constant independent of $n$. Moreover, if $r=r(n)=2$ for all $n$,
condition $(O_p)$ is also satisfied, since the measure $|A_{n,r}|$ of the largest of 
the atoms $A_{n,1},\ldots,A_{n,r}$ does not appear in $(O_p)$.

\begin{thm}\label{thm:uncond_more_than_two}
Let $1<p<\infty$. Then, the following assertions are true.
\begin{enumerate}
	\item If $(\psi_{A_n})$ satisfies the estimate 
			\begin{equation}\label{eq:uncond_plus}
				\Big\| \sum_n \pm \langle f,\psi_{A_n}\rangle \psi_{A_n} \Big\|_p \leq C \|f\|_p,\qquad f\in L^p,
			\end{equation}
			for some positive  constant $C$ independent of $f$ and the choice of signs, then condition $(O_p)$ is satisfied.
	\item 
	If condition $(O_1)$ is satisfied, 
			then $(\psi_{A_n})$ satisfies estimate \eqref{eq:uncond_plus} for 
			some positive  constant $C$ independent of $f$ and the choice of signs.

	\item If condition $(O_p)$ is satisfied, then the renormalized system $(\psi_{A_n}^p)$ given by 
			$\psi_{A_n}^p = \psi_{A_n}/\|\psi_{A_n}\|_p$ is democratic in $L^p$.
\end{enumerate}
\end{thm}

Statement (2) in particular gives that $(O_1)$ implies unconditionality of $(\psi_{A_{n}})$ in $L^p$.

\begin{proof}
	First consider statement (1). If $(\psi_{A_n})$ satisfies \eqref{eq:uncond_plus}, we 
	have in particular that there exists a positive  constant $C$ such that for all $n$ and all $f\in L^p$
	\[
		|\langle f,\psi_{A_n}\rangle|\cdot \|\psi_{A_n}\|_p \leq C\|f\|_p.
	\]
	This condition is equivalent to the fact that for all $n$
	\[
		\|\psi_{A_n}\|_{p'}	\cdot\|\psi_{A_n}\|_p \leq C,
	\]
	with $p' = p/(p-1)$ and the same positive  constant $C$.
	Inserting \eqref{eq:psiequiv} in this inequality implies condition $(O_p)$. 

	Next we prove statement (2). 
	Observe that for fixed $n$ and $f\in L^p$ we have that 
	\[
		\langle f,\psi_{A_n}\rangle \psi_{A_n} 	= \sum_{\mu,\nu=1}^{r-1} \sum_{\ell,m} \overline{a_{n,\mu,\ell}} a_{n,\nu,m}  \langle f,\psi_{n,\mu,\ell}\rangle \psi_{n,\nu,m}.
	\]
	Using the square functions, the fact that $(\psi_{n,j}') = (\psi_{n,\mu,\ell})$
	satisfies estimate \eqref{eq:uncond_plus_est}  and Lemma~\ref{lem:diff_square}, it is enough to show
	 the following estimate for establishing \eqref{eq:uncond_plus}: 
	\begin{equation}\label{eq:uncond_more_than_two}
	\begin{aligned}
		\bigg\| \bigg( \sum_n\sum_{\nu=1}^{r-1}\sum_m  \Big| \sum_{\mu=1}^{r-1} &\sum_{\ell}\pm \overline{a_{n,\mu,\ell}} a_{n,\nu,m}\langle f,\psi_{n,\mu,\ell}\rangle \Big|^2 \frac{\charfun_{A_{n,\nu}}}{|A_{n,\nu}|}	\bigg)^{1/2} \bigg\|_p  \\
		&\qquad \lesssim \bigg\|\bigg(\sum_n \sum_{\mu=1}^{r-1}  \sum_\ell |\langle f,\psi_{n,\mu,\ell}\rangle|^2
		\frac{\charfun_{A_{n,\mu}}}{|A_{n,\mu}|} \bigg)^{1/2} \bigg\|_p.
	\end{aligned}
	\end{equation}
	To this end, we investigate, for fixed $n$, fixed $\mu,\nu = 1,\ldots, r-1$ and fixed $\ell,m$ the 
	expression 
	\[
		|a_{n,\mu,\ell} a_{n,\nu,m}|^2 	\frac{\charfun_{A_{n,\nu}}}{|A_{n,\nu}|}
	\]
	and consider the trivial $\sigma$-algebra $\mathscr G =\mathscr G_{n,\mu,\nu}= \{\emptyset, A_{n,\mu}\cup A_{n,\nu}\}$ on $ A_{n,\mu}\cup A_{n,\nu}$.
	We observe 
	\[
		\mathbb E_{\mathscr G} \charfun_{A_{n,\mu}} = \frac{|A_{n,\mu}|}{|A_{n,\mu}| + |A_{n,\nu}|} \charfun_{A_{n,\mu}\cup A_{n,\nu}}
	\]
	and thus we write 
	\[
		\charfun_{A_{n,\mu}\cup A_{n,\nu}}	 = \charfun_{A_{n,\mu}\cup A_{n,\nu}}^2 = 
		\frac{ (|A_{n,\mu}| + |A_{n,\nu}|)^2}{|A_{n,\mu}|} \Big( \mathbb E_{\mathscr G} \frac{\charfun_{A_{n,\mu}}}{|A_{n,\mu}|^{1/2}}\Big)^2.
	\]
	Then we estimate 
	\begin{align*}
		|a_{n,\mu,\ell} a_{n,\nu,m}|^2 	\frac{\charfun_{A_{n,\nu}}}{|A_{n,\nu}|}  &= 
			|a_{n,\mu,\ell} a_{n,\nu,m}|^2 	\Big( \frac{\charfun_{A_{n,\nu}}}{|A_{n,\nu}|^{1/2}}\Big)^2
			\leq |a_{n,\mu,\ell} a_{n,\nu,m}|^2 	\Big( \frac{\charfun_{A_{n,\mu}\cup A_{n,\nu}}}{|A_{n,\nu}|^{1/2}}\Big)^2 \\
			& = 
			|a_{n,\mu,\ell} a_{n,\nu,m}|^2 	\frac{ (|A_{n,\mu}| + |A_{n,\nu}|)^2}{|A_{n,\mu}| |A_{n,\nu}|} 
			\Big(\mathbb E_{\mathscr G} \frac{\charfun_{A_{n,\mu}}}{|A_{n,\mu}|^{1/2}} \Big)^2  \\
			&\leq 4C^2 	\Big(\mathbb E_{\mathscr G} \frac{\charfun_{A_{n,\mu}}}{|A_{n,\mu}|^{1/2}} \Big)^2,
	\end{align*}
	where the last inequality follows from $(O_1)$.
	Inserting this inequality in the left-hand side of inequality \eqref{eq:uncond_more_than_two}
	and using Stein's martingale inequality \eqref{eq:steins_inequality} for the $\sigma$-algebras 
	$\mathscr G_{n,\mu,\nu}$ gives the right-hand side of \eqref{eq:uncond_more_than_two}.
	This finishes the proof of statement (2).

	In order to prove statement (3), we first show that under condition $(O_p)$, we have the inequality 
	\begin{equation}\label{eq:demo_more_than_two}
		\sum_\ell |a_{n,\mu,\ell}|  |A_{n,\mu}|^{1/p - 1/2} \lesssim  	 |A_{n,\mu_0}|^{1/p - 1/2}
	\end{equation}
	for all $\mu = 1,\ldots, r-1$,
		where $\mu_0=\mu_0(n)\in \{1,\ldots,r-1\}$ is chosen so that 
		$ \sum_\ell |a_{n,\mu_0,\ell}| = \max_{\mu=1,\ldots,r-1}\sum_\ell |a_{n,\mu,\ell}|$.
		Without restriction, assume that $p\leq 2$ so that $1/p - 1/2 \geq 0$.
	Observe that due to the assumption $\sum_{\mu=1}^{r-1}	\sum_\ell |a_{n,\mu,\ell}|^2=1$ 
	we have $\sum_\ell |a_{n,\mu_0,\ell}| \simeq 1$.
	We distinguish two cases for the value of $\mu$.

	\textsc{Case 1:  } $\mu\leq \mu_0$. In this case we have $|A_{n,\mu}| \leq |A_{n,\mu_0}|$ and 
	condition $(O_p)$ implies
	\[
		\sum_{\ell} |a_{n,\mu,\ell}| |A_{n,\mu}|^{1/p - 1/2} \lesssim  
		\frac{|A_{n,\mu}|^{2(1/p - 1/2)}}{|A_{n,\mu_0}|^{1/p - 1/2}} \leq |A_{n,\mu}|^{1/p - 1/2}\leq |A_{n,\mu_0}|^{1/p - 1/2}.
	\]
	proving \eqref{eq:demo_more_than_two} for $\mu\leq \mu_0$. 

	\textsc{Case 2: } $\mu > \mu_0$. In this case we have $|A_{n,\mu_0}| \leq |A_{n,\mu}|$
	and condition $(O_p)$ implies 
	\begin{align*}
		\sum_\ell |a_{n,\mu,\ell}| |A_{n,\mu}|^{1/p - 1/2} 
		\lesssim 
		|A_{n,\mu_0}|^{1/p-1/2},
	\end{align*}
	proving \eqref{eq:demo_more_than_two} also for $\mu>\mu_0$.

	Therefore, by \eqref{eq:psiequiv} and \eqref{eq:demo_more_than_two},
	\begin{equation}\label{eq:psi_A_n}
		 \| \psi_{A_n} \|_p \simeq \sum_{\mu=1}^{r-1} |A_{n,\mu}|^{1/p - 1/2} \sum_\ell |a_{n,\mu,\ell}| \simeq |A_{n,\mu_0}|^{1/p - 1/2}.
	\end{equation}
	Now, we write 
	\[
		\psi_{A_n}^p = \sum_{\mu=1}^{r-1} \sum_\ell  b_{n,\mu,\ell} \psi_{n,\mu,\ell}^p 		
	\]
	with 
	\[
		b_{n,\mu,\ell} = a_{n,\mu,\ell} \frac{\|\psi_{n,\mu,\ell}\|_p }{\|\psi_{A_n}\|_p}		
		\qquad \text{and} \qquad \psi_{n,\mu,\ell}^p = \frac{\psi_{n,\mu,\ell}}{\|\psi_{n,\mu,\ell}\|_p}.
	\]
	Note that for all $\mu = 1,\ldots,r-1$, 
	\[
		\sum_\ell |b_{n,\mu,\ell}| 	= \sum_\ell |a_{n,\mu,\ell}| \frac{\|\psi_{n,\mu,\ell}\|_p}{\|\psi_{A_n}\|_p}	
		\simeq \sum_\ell |a_{n,\mu,\ell}| \frac{|A_{n,\mu}|^{1/p - 1/2}}{|A_{n,\mu_0}|^{1/p - 1/2}}
	\]
	by \eqref{eq:psi_A_n}.
	Therefore, \eqref{eq:demo_more_than_two} implies 
		$\sum_\ell |b_{n,\mu,\ell}| \lesssim 1$
	for all $\mu=1,\ldots,r-1$ and $\sum_\ell |b_{n,\mu_0,\ell}|\simeq 1$.
	Let now $\Lambda$ be a finite set of positive integers. Then, this bound on the 
	numbers $(b_{n,\mu,\ell})$ and the unconditionality and democracy of the functions 
	$(\psi_{n,\mu,\ell}^p)$ gives us 
	\[
		\big\| \sum_{n\in\Lambda} \psi_{A_n}^p \big\|_p = 
		\Big\| \sum_{n\in\Lambda} \sum_{\mu=1}^{r-1} \sum_\ell b_{n,\mu,\ell} \psi_{n,\mu,\ell}^p \Big\|_p
		\gtrsim \Big\| \sum_{n\in\Lambda}  \sum_\ell |b_{n,\mu_0,\ell}| \psi_{n,\mu_0,\ell}^p \Big\|_p
		\gtrsim (\card \Lambda)^{1/p}.
	\] 
	Similarly, the converse inequality $\| \sum_{n\in\Lambda} \psi_{A_n}^p\|_p \lesssim (\card \Lambda)^{1/p}$
	is a consequence of $\sum_\ell |b_{n,\mu,\ell}| \lesssim 1$, the unconditionality 
	and democracy of the functions $(\psi_{n,\mu,\ell}^p)$, where we also use that 
	the number $r=r(n)$ is uniformly bounded in $n$ and the number of terms 
	in the sum over $\ell$ is bounded by the dimension of the space $S$.
\end{proof}

\begin{prop}\label{prop:suff_cond_orth}
	Suppose that  the functions $\psi_{A_n}=\sum_{\mu=1}^{r-1} \sum_\ell  a_{n,\mu,\ell} \psi_{n,\mu,\ell}$
satisfy the pointwise inequality 
	\[
		|\psi_{A_n}| \lesssim \frac{ |A_{n,1}|^{1/2}}{|A_{n,\kappa}|} \qquad\text{on } A_{n,\kappa}
	\]
	for each $\kappa=1,\ldots, r$.

	Then, for each $1<p<\infty$, the system $(\psi_{A_n})$ satisfies \eqref{eq:uncond_plus} and 
	the renormalized system $(\psi_{A_n}^p)$ is democratic in $L^p$.
\end{prop}
\begin{proof}
We show condition $(O_1)$, i.e., we show 
	for each $\mu,\nu=1,\ldots,r-1$ the inequality
	\begin{equation}\label{eq:to_prove}
		\sum_{\ell} |a_{n,\mu,\ell}|	\sum_{m} |a_{n,\nu,m}| |A_{n,\mu}|^{1/2} 
		|A_{n,\nu}|^{-1/2} \leq C,
	\end{equation}
	for some positive  constant $C$ independent of $n$.
	For each $\mu=1,\ldots,r-1$ and each $\ell$, we observe that 
	the pointwise estimates for $\psi_{A_n}$ in the assumptions and the pointwise estimates 
	\eqref{eq:est_psi_more_than_two} for $\psi_{n,\mu,\ell}$ imply 
	\begin{align*}
		|a_{n,\mu,\ell}| &= |\langle \psi_{A_n}, \psi_{n,\mu,\ell}\rangle|	
		\leq \sum_{\kappa = \mu}^r  \int_{A_{n,\kappa}} |\psi_{A_n}|\cdot |\psi_{n,\mu,\ell}| \\
		& \lesssim |A_{n,\mu}| \cdot \frac{|A_{n,1}|^{1/2}}{|A_{n,\mu}|} \cdot |A_{n,\mu}|^{-1/2} + 
		\sum_{\kappa=\mu+1}^{r} |A_{n,\kappa}| \cdot \frac{|A_{n,1}|^{1/2}}{|A_{n,\kappa}|}\cdot \frac{|A_{n,\mu}|^{1/2}}{|A_n|} \\
		& =  \frac{|A_{n,1}|^{1/2}}{|A_{n,\mu}|^{1/2}} +  \frac{(r-\mu)\cdot |A_{n,1}|^{1/2} |A_{n,\mu}|^{1/2}}{|A_n|} \lesssim 
		\frac{|A_{n,1}|^{1/2}}{|A_{n,\mu}|^{1/2}}.
	\end{align*}
	Inserting this inequality in the left-hand side of \eqref{eq:to_prove}, we obtain \eqref{eq:to_prove}
	since $|A_{n,1}|\leq |A_{n,\mu}|$ for all $\mu=1,\ldots, r-1$.
	Thus, Theorem~\ref{thm:uncond_more_than_two}
	now implies that \eqref{eq:uncond_plus} is true for $(\psi_{A_n})$ and $(\psi_{A_n}^p)$ is democratic in $L^p$.
\end{proof}

We now use this result to give an explicit example of an orthonormal system of functions that 
does not fall in the category of functions of the form $(\psi_{n,j}')$ as considered 
above, but is still unconditional and democratic in $L^p$ for $1<p<\infty$.
The example we want to consider consists of local tensor products of functions.
Assume that we are given probability spaces $(\Omega_\delta,\mathscr F^\delta,\mathbb P^\delta)$
for $\delta = 1,\ldots, d$. Assume that $(\mathscr F_n)$ is a filtration on $\Omega_1\times \cdots\times \Omega_d$
with 
$\mathscr F_0 = \{\emptyset,\Omega_1\times\cdots\times \Omega_d\}$ and $\mathscr F_{n}$ is 
generated by $\mathscr F_{n-1}$ and the partition of one atom $A_n=A_{n}^1\times\cdots\times A_n^d$ of $\mathscr F_{n-1}$
into the atoms $A_{n,m_1}^1\times\cdots\times  A_{n,m_d}^d$
with $m_\delta\in \{1,2\}$ for each $\delta = 1,\ldots,d$ and the same ordering 
\[
	|A^\delta_{n,1}| \leq |A^\delta_{n,2}|,\qquad \delta = 1,\ldots,d
\]
as above in each direction $\delta$.
Denote by $\mathscr A_n$ the atoms of $\mathscr F_n$ and  $\mathscr A = \cup_n \mathscr A_n$.
Let $p_\delta : \Omega_1\times \cdots\times \Omega_d \to \Omega_\delta$ be the projection onto 
the $\delta$-th coordinate. Then, let $\mathscr G^\delta := p_\delta(\mathscr A)$ be the collection of 
subsets of $\Omega_\delta$ that appear as the $\delta$-th coordinate projection of some atom $A\in\mathscr A$.
Moreover, assume that we are given finite-dimensional 
subspaces  $S^\delta \subset L^\infty(\Omega_\delta)$ that satisfy \eqref{eq:L1Linfty} for each $G_\delta\in\mathscr G^\delta$, i.e.,
\begin{equation}\label{eq:L1Linfty_delta}
			|\{ \omega_\delta\in G_\delta : |f(\omega_\delta)| \geq c_1\|f\|_{G_\delta}
		\}| \geq c_2|G_\delta|,\qquad f\in S^\delta.
\end{equation}
Then we consider the tensor product space $S := S^1 \otimes \cdots \otimes S^d$ on $\Omega_1\times \cdots \times \Omega_d$.
We now show that $S$ satisfies \eqref{eq:L1Linfty} for some constants $c_1',c_2'\in(0,1]$ and 
for every $G = G_1\times \cdots \times G_d$ with $G_\delta\in\mathscr G^\delta$, i.e.,
\begin{equation}\label{eq:L1Linfty_tensor}
			|\{ \omega\in G : |f(\omega)| \geq c_1'\|f\|_{G}
		\}| \geq c_2'|G|,\qquad f\in S.
\end{equation}
It is enough to show that for each $G$ of this form and each $f\in S$, the mean value of $|f|$ on $G$
is comparable to the $L^\infty$-norm of $f$ on $G$. Let $d\geq 2$ and $f \in S$. Let $x\in G_1\times \cdots\times G_{d-1}$ 
and $y\in G_d$ and let $f_x(y) = f_y(x) = f(x,y)$. The latter functions are measurable in the 
respective measure spaces. Then, using the notation $G' = G_1\times \cdots \times G_{d-1}$,  we calculate by Fubini's theorem
and the fact that on the coordinate projection $G_d$, mean values of positive functions in $S^d$ are 
comparable to their $L^\infty$-norms by \eqref{eq:L1Linfty_delta},
\begin{align*}
\frac{1}{|G|} \int_G |f(x,y)| \dif (x,y) &= \frac{1}{|G'|}\int_{G'} \frac{1}{|G_d|} \int_{G_d} |f(x,y)|\dif y \dif x \\
&\gtrsim  \frac{1}{|G'|}\int_{G'} \|f_x\|_{G_d} \dif x \geq \Big\| \frac{1}{|G'|}\int_{G'} |f_y|\dif x \Big\|_{G_d}.
\end{align*}
By induction, this implies that the mean of $|f|$ on $G$ is comparable to the $L^\infty$-norm of $f$ on $G$.
Thus, we have established \eqref{eq:L1Linfty_tensor} for some constants $c_1',c_2'\in(0,1]$.

Then we get, for each $n$ and in each direction $\delta=1,\ldots,d$  
an orthonormal basis 
$(\psi_{n,0,j_\delta}^\delta)_{j_\delta}$ of $S^\delta(A_n^\delta)$
and an orthonormal basis  
$(\psi_{n,1,j_\delta}^\delta)_{j_\delta}$ of the orthogonal complement of $S^\delta(A_n^\delta)$
in $S^\delta(A_{n,1}^\delta)\oplus S^\delta(A_{n,2}^\delta)$.
Then we consider the orthonormal system $(\psi_{n,\mu,j})$ for indices $\mu=(\mu_1,\ldots,\mu_d)$ and 
$j=(j_1,\ldots,j_d)$ with $\mu\neq 0$ consisting of the functions 
\begin{equation}\label{eq:psi_tensor}
	\psi_{n,\mu,j} = \psi_{n,\mu_1,j_1}^1 \otimes\cdots\otimes \psi_{n,\mu_d,j_d}^d.	
\end{equation}
For fixed $n$, those functions are orthogonal to the space $S^1(A_n^1)\otimes\cdots\otimes S^d(A_n^d)$.

We use Lemma~\ref{lem:ess_estimate}
to obtain the following estimate for every $\delta=1,\ldots,d$ 
and every $j_\delta$,
\begin{equation}
		|\psi_{n,1,j_\delta}^\delta| \lesssim \begin{cases}
			0,& \text{on } (A_{n}^\delta)^c, \\
			|A_{n,1}^\delta|^{-1/2},& \text{on } A_{n,1}^\delta, \\
			|A_{n,1}^\delta|^{1/2}/|A_{n,2}^\delta|,& \text{on } A_{n,2}^\delta.
		\end{cases}
	\end{equation}
	In the case $\mu_\delta=0$ we have the estimate 
	\[
		|\psi_{n,0,j_\delta}^\delta| \lesssim |A_n^\delta|^{-1/2}\qquad\text{on } A_n^\delta,
	\]
	which is a consequence of \eqref{eq:L1Linfty_delta}.
	Assume $\mu_1 = \cdots = \mu_s = 0$ and $\mu_{s+1} = \cdots = \mu_d = 1$ and $m_{s+1} = \cdots m_t = 1$, $m_{t+1} = \cdots = m_d =2$.
	Then, we estimate on the set  $A_n^1 \times \cdots A_n^s\times A_{n,m_{s+1}}^{s+1} \times \cdots \times A_{n,m_{d}}^d$
	as follows: 
	\begin{equation}\label{eq:pw_tensor}
	\begin{aligned}
		|\psi_{n,\mu,j}| &\lesssim (|A_n^1|\cdots|A_n^s| |A_{n,1}^{s+1}| \cdots |A_{n,1}^t| )^{-1/2} 
		\Big( \frac{|A_{n,1}^{t+1}|^{1/2}}{|A_{n,2}^{t+1}|} \cdots \frac{|A_{n,1}^{d}|^{1/2}}{|A_{n,2}^{d}|} \Big) \\
		& = \frac{|A_n^1 \times \cdots \times A_n^s \times A_{n,1}^{s+1} \times \cdots \times A_{n,1}^d|^{1/2} }
				{|A_{n}^1\times \cdots\times A_n^s \times A_{n,1}^{s+1}\times \cdots \times A_{n,1}^t \times A_{n,2}^{t+1}\times \cdots \times A_{n,2}^d|}.
	\end{aligned}
\end{equation}
For each fixed value of $\mu \in \{0,1\}^d\setminus\{0\}$ and $j$ we now define a filtration on $\Omega_1\times \cdots \times \Omega_d$
so that we can use the pointwise estimate \eqref{eq:pw_tensor} to show the assumption of Proposition~\ref{prop:suff_cond_orth}
for the sequence of functions $(\psi_{n,\mu,j})_n$ defined in \eqref{eq:psi_tensor}. 
Up to permutation of variables, we assume without restriction that $\mu_1 = \cdots = \mu_s = 0$ and $\mu_{s+1} = \cdots = \mu_d = 1$
for some $s\in\{1,\ldots,d\}$.
Then we introduce the intermediate $\sigma$-algebra $\mathscr F_n^\mu$ satisfying $\mathscr F_{n-1} \subseteq \mathscr F_n^\mu \subseteq \mathscr F_n$ 
by defining the atoms of $\mathscr F_{n}^\mu$ to be the sets 
\[
	A_n^1\times \cdots \times A_n^s	\times A_{n,m_{s+1}}^{s+1} \times \cdots \times A_{n,m_d}^{d},\qquad (m_{s+1},\ldots,m_d)\in \{1,2\}^{d-s}.
\]
Then, the function $\psi_{n,\mu,j}$ is contained in the orthogonal complement of $S(\mathscr F_{n-1})$
in $S(\mathscr F_{n}^\mu)$. Consider the filtration $(\widetilde{\mathscr F}_n)$ consisting of the $\sigma$-algebras $(\mathscr F_n)_n$ and $(\mathscr F_n^\mu)_n$
and let $\widetilde{\mathscr A}$ be the collection of atoms of $(\widetilde{\mathscr F}_n)$.
We showed above that \eqref{eq:L1Linfty} is satisfied for every atom $A\in \widetilde{\mathscr A}$ and every function $f\in S$.
Observe also that estimate \eqref{eq:pw_tensor} implies the assumption of Proposition~\ref{prop:suff_cond_orth} 
for the filtration $(\widetilde{\mathscr F}_n)$.
Therefore, we can use our theory, in particular Proposition~\ref{prop:suff_cond_orth}, to deduce that 
for every $f\in L^p$, we have the estimate 
\[
\Big\| \sum_n \pm\langle f,\psi_{n,\mu,j}\rangle \psi_{n,\mu,j}\Big\|_p	\leq C\| f\|_p,
\]
where the positive  constant $C$ depends on $p$ and on $\mu,j$, but not on $f$. Since there are only 
uniformly finitely many values of $\mu,j$, we deduce that the orthonormal system $(\psi_{n,\mu,j})_{n,\mu,j}$ 
is unconditional in $L^p$.

Additionally, Proposition~\ref{prop:suff_cond_orth} implies that for each $\mu,j$, the renormalized system 
$(\psi_{n,\mu,j}^p)_n$ is democratic in $L^p$. Choose a finite set $\Lambda$ of values $(n,\mu,j)$.
Since $(\psi_{n,\mu,j}^p)_{n,\mu,j}$ is unconditional in $L^p$, we get 
\begin{align*}
	\Big\| \sum_{(n,\mu,j)\in\Lambda} \psi_{n,\mu,j}^p\Big\|_p \simeq \Big\| \Big(\sum_{(n,\mu,j)\in \Lambda} (\psi_{n,\mu,j}^p)^2\Big)^{1/2} \Big\|_p
	\simeq \Big\| \sum_{\mu,j}\Big(\sum_{n: (n,\mu,j)\in\Lambda} (\psi_{n,\mu,j}^p)^2\Big)^{1/2} \Big\|_p,
\end{align*}
where the latter equivalence is a consequence of the fact that the number of indices $(\mu,j)$ is 
uniformly bounded in $n$. Thus, there exist values $(\mu_0,j_0)$ such that we have 
\[
	\Big\| \sum_{(n,\mu,j)\in\Lambda} \psi_{n,\mu,j}^p\Big\|_p \simeq 
\Big\| \Big(\sum_{n: (n,\mu_0,j_0)\in\Lambda} (\psi_{n,\mu_0,j_0}^p)^2\Big)^{1/2} \Big\|_p\simeq (\card \Lambda)^{1/p},
\]
implying the democracy of the system $(\psi_{n,\mu,j}^p)_{n,\mu,j}$.

\subsection*{Acknowledgments}
	M. Passenbrunner is supported by the Austrian Science Fund FWF, project P32342.
The authors would like to thank the anonymous referees for carefully reading the paper and for their remarks.

\bibliographystyle{plain}
\bibliography{local}

\end{document}